\documentclass[english]{amsart}
\usepackage[T1]{fontenc}
\usepackage[utf8]{inputenc}
\usepackage{babel}
\usepackage{prettyref}
\usepackage{varioref}
\usepackage{amsmath}
\usepackage{amssymb}
\usepackage{graphicx}
\usepackage[percent]{overpic}
\usepackage{pict2e}
\usepackage[unicode=true,pdfusetitle,bookmarks=true,bookmarksnumbered=false,bookmarksopen=false,breaklinks=false,pdfborder={0 0 0},backref=false,colorlinks=false]{hyperref}
\usepackage{csquotes}
\usepackage{enumitem}
\usepackage{todonotes}

\theoremstyle{plain}
\newtheorem{thm}{\protect\theoremname}
\newtheorem{prop}[thm]{\protect\propositionname}
\newtheorem{lem}[thm]{\protect\lemmaname}

\theoremstyle{definition}
\newtheorem{defn}[thm]{\protect\definitionname}
\newtheorem{example}[thm]{\protect\examplename}

\numberwithin{thm}{section}

\usepackage{tikz}
\usetikzlibrary{arrows}

\providecommand{\corollaryname}{Corollary}
\providecommand{\definitionname}{Definition}
\providecommand{\examplename}{Example}
\providecommand{\lemmaname}{Lemma}

\providecommand{\propositionname}{Proposition}
\providecommand{\theoremname}{Theorem}

\providecommand{\sectionname}{Section}

\newrefformat{def}{\definitionname~\ref{#1}}
\newrefformat{sec}{Section~\ref{#1}}
\newrefformat{cor}{\corollaryname~\ref{#1}}
\newrefformat{prop}{\propositionname~\ref{#1}}
\newrefformat{ex}{\examplename~\ref{#1}}
\newrefformat{sub}{\sectionname~\ref{#1}}
\newrefformat{app}{\appendixname~\ref{#1}}
\newrefformat{fig}{Figure~\vref{#1}}

\newcommand{\Z}{\mathbb Z}
\newcommand{\R}{\mathbb R}

\newcommand{\interior}{\operatorname{int}}
\newcommand{\id}{\operatorname{id}}
\newcommand{\im}{\operatorname{im}}
\newcommand{\tr}{\operatorname{tr}}

\newcommand{\Ric}{\operatorname{Ric}}
\newcommand{\Riem}{\operatorname{Riem}}

\newcommand{\D}{\mathfrak D}
\newcommand{\mcM}{\mathcal M}
\newcommand{\M}{\,\overline{\!M}}
\newcommand{\Mprime}{\,\overline{\!M'}}
\newcommand{\TdM}{\partial T\M}
\newcommand{\TdMprime}{\partial T\Mprime}
\newcommand{\oTM}{\widetilde TM}
\newcommand{\oNM}{\widetilde NM}
\newcommand{\oNpM}{\widetilde N_pM}
\newcommand{\NdM}{\partial N\M}

\newcommand{\q}[2]{#1\big/\!\!#2}
\newcommand{\qextraspace}[2]{#1\big/#2}

\newcommand{\stepheader}[1]{\vspace{1em}\noindent\textbf{#1}\\*[.2em]}

\makeatletter
\vref@addto\extrasenglish{%

}
\makeatother

\begin{document}

\title[Broken causal lens rigidity and sky shadow rigidity]{Broken causal lens rigidity and sky shadow rigidity of Lorentzian manifolds}
\author{Eric Larsson}
\address{
	Department of Mathematics\\
	KTH Royal Institute of Technology\\
	SE-100 44, Stockholm\\
	Sweden
}
\email[Eric Larsson]{ericlar@kth.se}

\maketitle

\begin{abstract}
We prove that the topology, smooth structure, and metric of a compact Lorentzian manifold with boundary is uniquely determined by data at the boundary. The data consists of the lengths and directions of future-directed once-broken geodesics connecting points on the boundary, which are first timelike and then lightlike.
This requires the strong causality condition and a weak convexity assumption, but it holds without any assumptions about conjugate points.
With an additional convexity assumption we prove the analogous statement for future-directed once-broken timelike geodesics.

If there are no conjugate points and lightlike geodesics never refocus, the analogous data using lightlike geodesics and once-broken lightlike geodesics may be used to reconstruct the manifold up to a conformal factor. This is a corollary of a result which shows that the conformal class is determined by the collection of sets of future-directed lightlike vectors at the boundary which give geodesics which all intersect in a single point. 
\end{abstract}

\setcounter{tocdepth}{1}
\tableofcontents

\section{Introduction}
Consider a compact region \(\M\) in a semi-Riemannian manifold. Suppose that we, from outside of \(\M\), are allowed to send test particles following geodesics into \(\M\) and observe where and in which direction they exit \(\M\). How much of the geometry of \(\M\) can we determine? This is the problem of \enquote{scattering rigidity}, and the data gathered from such observations is called \enquote{scattering data}. If the lengths of the geodesics is also part of the data, then the data is called \enquote{lens data}. There are several results on the problems of scattering rigidity and lens rigidity in Riemannian signature (see~\cite[Section 1.2]{CrokeWen15} for a list with references), and a few results in Lorentzian signature (see~\cite{AnderssonDahlHoward96}).

We will consider similar data, where the test particles are allowed to change direction once (so that they follow broken geodesics with at most one breakpoint). We call such data \enquote{broken scattering data}, or \enquote{broken lens data} if the data also contains the lengths of the geodesics. The problem of reconstructing a Riemannian manifold from broken lens data was considered by Kurylev, Lassas, and Uhlmann in~\cite{KurylevLassasUhlmann10}. They proved that the broken lens data\footnote{What we call broken lens data is in the terminology of~\cite{KurylevLassasUhlmann10} instead called \enquote{the broken scattering relation}. We choose our terminology to agree with the terminology used for instance by Croke and Wen and by Andersson, Dahl, and Howard; see~\cite{CrokeWen15},~\cite{AnderssonDahlHoward96} and the references contained therein.} of a compact, connected Riemannian manifold with nonempty boundary determines the isometry type of the manifold. 
We will consider different versions of the problem for Lorentzian manifolds. We choose to consider only geodesic segments which are future-directed and causal.
This restriction makes the problem very different from the Riemannian case. For instance it excludes broken geodesics which consist of following a single geodesic segment twice in opposite directions. Such broken geodesics were used in~\cite{KurylevLassasUhlmann10}.

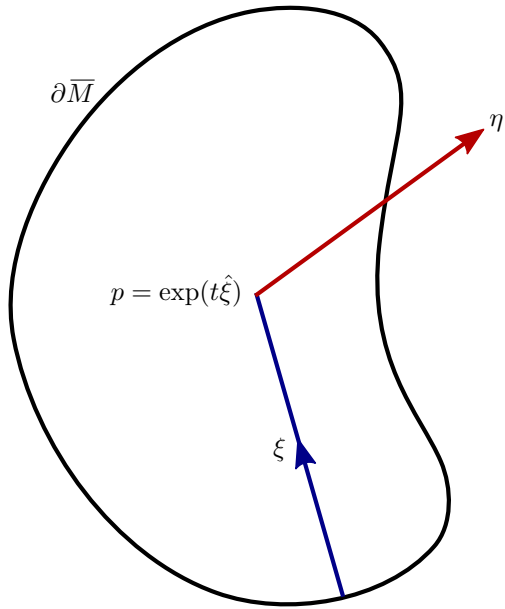
\begin{figure}
 \centering{}
 \begin{overpic}[width=0.5\textwidth]{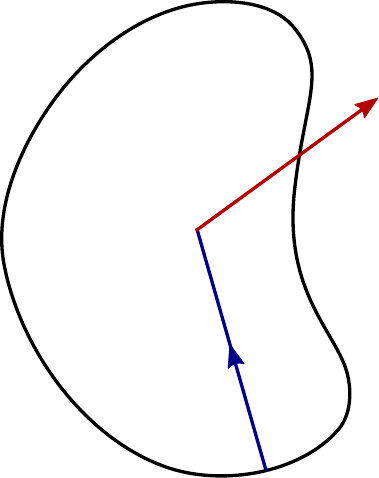}
  \put (17,51) {\(p = \exp(t\hat\xi)\)}
  \put (44,25) {\(\xi\)}
  \put (80,80) {\(\eta\)}
  \put (7,84) {\(\partial \M\)}
 \end{overpic}
 \caption{The triples \((\xi, t, \eta)\) with \(\xi\) timelike, \(\eta\) lightlike, and \(\hat \xi\) being the normalization of \(\xi\) to unit length related as in the picture, form the fireworks data.}
\label{fig:Fireworks data}
\end{figure}

There are still different possible versions of the problem, depending on the allowed causal type of the two parts of the broken geodesics. One option is to consider only lightlike broken geodesics. In this case, the lengths of geodesics are always zero so the lens rigidity problem is equivalent to the scattering rigidity problem. This is the version we consider in \prettyref{sec:Broken lightlike scattering rigidity}. Our solution to this problem is a corollary of the \enquote{sky shadow rigidity} discussed in \prettyref{sec:Sky shadow rigidity}.
The name \enquote{sky shadow} is inspired by the use of the word \enquote{sky} in~\cite{Low88} and~\cite[Chapter 1]{PenroseRindler1} to mean the set of lightlike geodesics through a point. A result about reconstructing the manifold structure and metric of a Lorentzian manifold from its space of skies can be found in~\cite{BautistaIbfortLafuente14}. For other questions about the space of lightlike geodesics of a Lorentzian manifold we refer to~\cite{Low06}.

Another option is to consider only timelike broken geodesics, as we do in \prettyref{sec:Broken timelike lens rigidity}. A solution to the rigidity problem with this data follows by a limiting argument from a solution to the lens rigidity problem where the data is based on broken geodesics which are first timelike and then lightlike (see \prettyref{fig:Fireworks data}). We call this latter problem \enquote{fireworks rigidity} and discuss it in \prettyref{sec:Fireworks rigidity}. The name \enquote{fireworks} refers to a possible physical interpretation: Suppose that we send in free-falling fireworks into the region \(\M\) which explode after a predetermined time, transmitting light in all directions. Then the fireworks themselves will follow timelike geodesics, and the light they transmit will follow lightlike geodesics. By observing the light at the boundary \(\partial \M\), we obtain precisely the broken lens data for once-broken geodesics which are at first timelike, and then lightlike.

Fireworks rigidity, broken timelike lens rigidity, and broken lightlike scattering rigidity could collectively be called \enquote{broken causal lens rigidity} since the geodesics involved are causal and future-directed. The broken lightlike scattering rigidity problem is included here since it is equivalent to the broken lightlike lens rigidity problem; lightlike geodesics have zero length, so including the curve length in the data gives no additional information. A similar problem, which we will not discuss further, is broken timelike scattering rigidity, where the data is similar to the broken timelike lens data but with the curve lengths omitted. 

In general, solutions to the rigidity problems we have mentioned tell us the following: \enquote{Suppose that two manifolds have the same data. Then the manifolds are necessarily the same.}
For our results, we impose rather weak restrictions on the two manifolds to be compared. 
The result in~\cite{KurylevLassasUhlmann10}, which also deals with data from broken geodesics, shares this feature, and in fact imposes only the conditions that the manifolds are compact, connected, and have nonempty boundary. 
This is in contrast with typical results about lens and scattering rigidity from unbroken geodesics. Such results are typically asymmetric in that they impose very strong conditions on one of the manifolds but not on the other. For instance,~\cite{Michel81} includes a result where one of the manifolds is a subdomain of a hemisphere. Similarly, a rigidity statement was shown in~\cite{BuragoIvanov10} (using the concept of \enquote{filling volume} from~\cite{Gromov83}) when one of the manifolds is a region in \(\R^n\) with a metric sufficiently close to the Euclidean metric. These papers deal with \enquote{boundary distance rigidity}, which is equivalent to lens rigidity when one of the manifolds is either \enquote{simple} or \enquote{strongly geodesically minimizing}. Details can be found in~\cite{Croke04}. Results about scattering rigidity often require even stronger conditions. One such scattering rigidity result is shown in~\cite{Wen15} under the assumption that the manifolds are two-dimensional and one of them is simple. This extends a previous result in~\cite{PestovUhlmann05} about lens rigidity under the same assumptions.

The corresponding problems for Lorentzian manifolds have received less attention. Lorentzian analogues of the boundary and lens rigidity results can be found in~\cite{AnderssonDahlHoward96} for Lorentzian \enquote{spacelike slabs} having spacelike boundary. Similar questions can be posed and answered for data on timelike hypersurfaces; for instance, in~\cite{LassasOksanenYang15} the restriction of the distance function to a timelike hypersurface was used to determine the \(C^\infty\)-jet of the metric on the hypersurface. 

\subsection{Main results}
For definitions of some of the terminology used here and later in the paper, see \prettyref{app:Assorted definitions}. Our first main result is \prettyref{thm:Fireworks data determines isometry class}. It tells us that if we know that a time-oriented Lorentzian manifold with boundary is strongly causal and compact, that it satisfies a weak convexity condition, and we know the fireworks data of the manifold, then we can determine the manifold up to isometry. Note that we do not assume that we know anything at all about the topology of the interior of the manifold. As a corollary we obtain \prettyref{thm:Broken timelike lens rigidity}, which says that we can determine a time-oriented Lorentzian manifold up to isometry if we know that it is strongly causal and satisfies a convexity condition, and we know its broken timelike lens data. The conditions that the manifold is compact and that the measurements are made at the boundary are not essential; the techniques used in the proof extend to some cases when the measurements are made on arbitrary \enquote{sufficiently large} hypersurfaces in manifolds which may or may not have boundary. We have not tried to make these conditions precise, and we will consider only the case when the manifold is compact and the measurements are made at the boundary.

\begin{figure}
 \centering{}
 \begin{overpic}[width=0.5\textwidth]{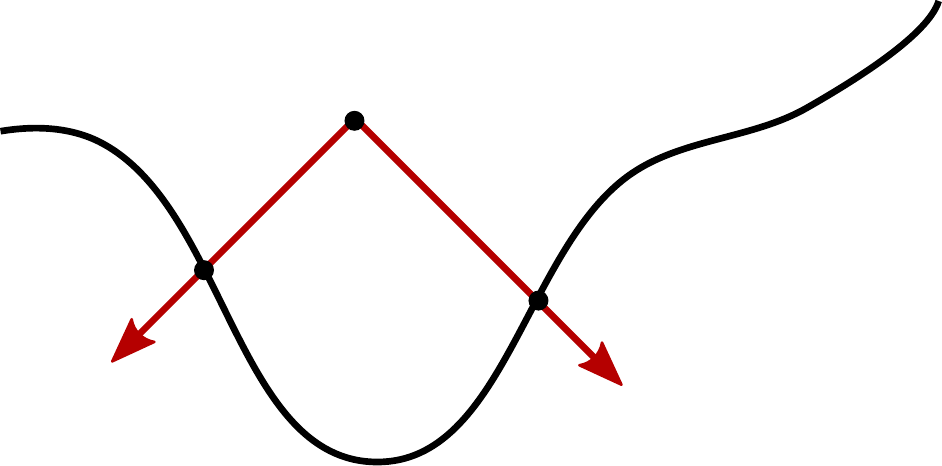}
  \put (37,40) {\(p\)}
  \put (6,10) {\(\eta_1\)}
  \put (67,9) {\(\eta_2\)}
  \put (65,25) {\(\partial \M\)}
 \end{overpic}
 \caption{The figure shows a region in the \((1+1)\)-dimensional Minkowski plane. The past sky shadow of \(p\) is the set of all nonzero vectors parallel to the lightlike vectors \(\eta_1\) and \(\eta_2\).}
\label{fig:Past sky shadow}
\end{figure}

\begin{figure}
 \centering{}
 \begin{overpic}[width=0.5\textwidth]{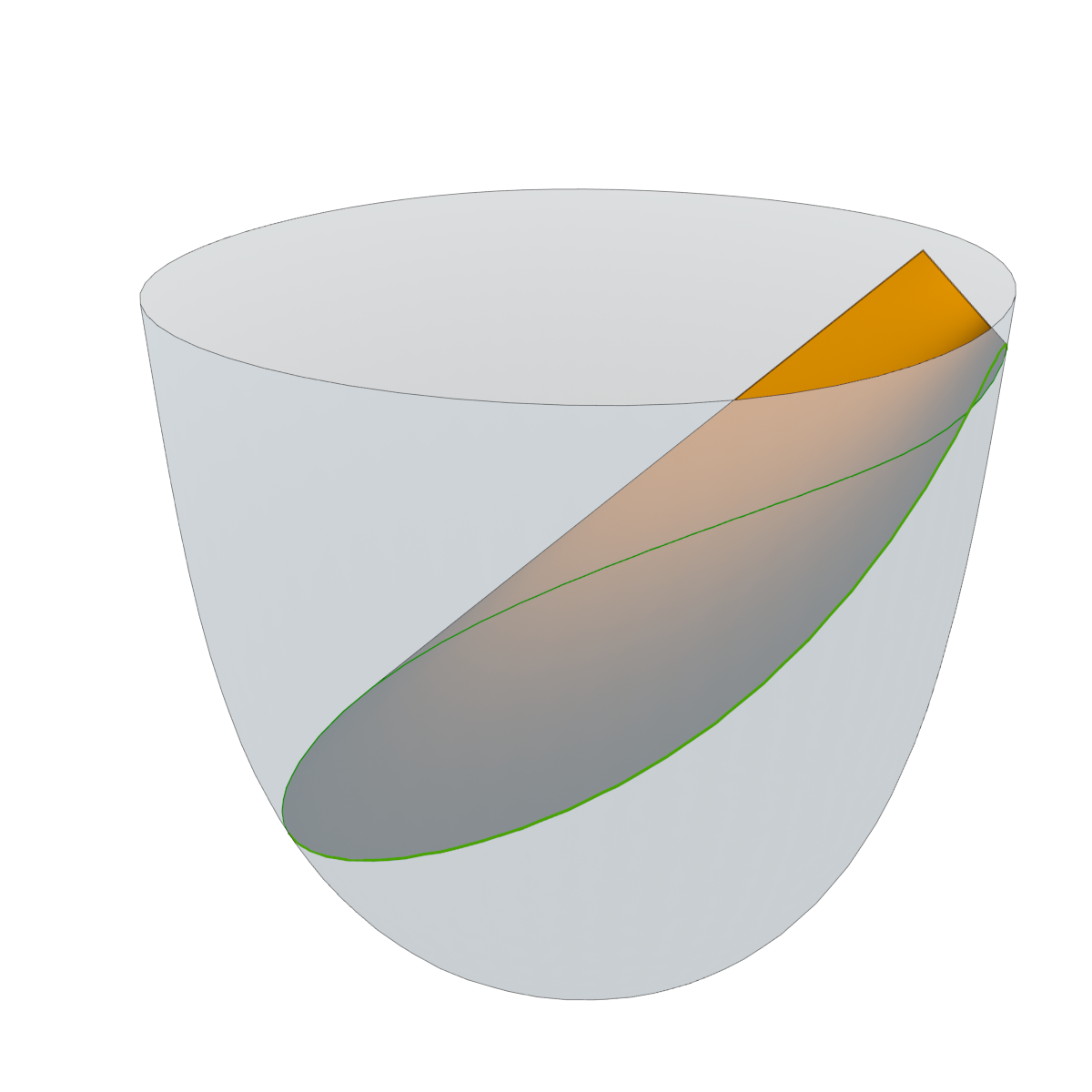}
  \put (85,79) {\(p\)}
  \put (62,26) {\(\Sigma\)}
  \put (18,48) {\(\partial \M\)}
 \end{overpic}
 \caption{The figure shows a region in the \((2+1)\)-dimensional Minkowski space. The past sky shadow of \(p\) is the set of all nonzero lightlike vectors based at the curve \(\Sigma\) and tangent to the cone.}
\label{fig:Past sky shadow 3D}
\end{figure}

Our second main result is \prettyref{thm:Past sky shadow data determines conformal class}, concerning what we call \enquote{sky shadow rigidity}. Following~\cite{Low88}, by a \enquote{sky} we mean the set of all lightlike geodesics through a point. The \enquote{sky shadow} of a point is the set of nonzero tangents of these geodesics which are based at the boundary of the manifold. More specifically, the \enquote{past sky shadow} (see \prettyref{fig:Past sky shadow} and \prettyref{fig:Past sky shadow 3D}) of a point \(p\) consists of all those lightlike vectors at the boundary such that the future-directed geodesics they generate pass through the point \(p\). This is similar to the \enquote{light observation sets} used in~\cite{KurylevLassasUhlmann15}, with the difference that the sky shadow is a subset of the tangent bundle instead of the manifold itself, and that we consider sky shadows on the boundary while the light observation sets are defined by intersections with an open neighborhood of a timelike geodesic segment. (See~\cite[Theorem 1.2]{KurylevLassasUhlmann15} for a reconstruction theorem involving these light observation sets.) The question is then, if we know the collection of all past sky shadows for a Lorentzian manifold with boundary, can we determine the manifold and the conformal class of the metric? (Since the data is invariant under conformal changes of the metric, we cannot hope to determine more than this.) We answer this question affirmatively in \prettyref{thm:Past sky shadow data determines conformal class}, under the condition that the manifold is known to be strongly causal, has sufficiently many lightlike geodesics transverse to its boundary, and has no lightlike geodesics with conjugate points. As a corollary we obtain \prettyref{thm:Broken lightlike scattering rigidity}, which tells us, under similar conditions, that the broken lightlike scattering data together with the unbroken lightlike scattering data is sufficient to determine the manifold and conformal class.
Again, the conditions that the manifold is compact and that the measurements are made at the boundary are not essential, but for simplicity we will only consider this case.

\subsection{Reconstruction of smooth structures}\label{sec:Reconstruction of smooth structures}
As part of the proofs, we will need to reconstruct the topology and smooth structure of a manifold from a set of data. A common way of reconstructing a smooth structure is to reconstruct suitable smooth coordinates (see for instance~\cite[Section 3.8]{KatchalovKurylevLassas} and~\cite[Section 5.1]{KurylevLassasUhlmann15}). We will use a different method. We first use the data to construct a smooth manifold \(\Omega\) and an appropriately chosen equivalence relation \(\sim\) on \(\Omega\). (The equivalence relation need not a priori make \(\q{\Omega}{\sim}\) a smooth manifold.) Given a smooth manifold \(M\) which realizes the data in question, we then construct a smooth surjective submersion \(\Theta \colon \Omega \to M\) which is such that it descends to a bijection \(\tilde\Theta \colon \q{\Omega}{\sim} \to M\). The important observation is that if there exists a smooth structure on \(\q{\Omega}{\sim}\) which makes the quotient map \(\Omega \to \q{\Omega}{\sim}\) a smooth surjective submersion, then this smooth structure is unique. This then determines the smooth structure on \(M\) in the following way: Suppose that we have two smooth manifolds \(M\) and \(M'\) realizing the same set of data. Then we get a commutative diagram as follows:
\[\begin{tikzpicture}[every node/.style={midway}]
\matrix[column sep={5em,between origins}, row sep={2em}] at (0,0) {
	; &
	\node(Omega) {\(\Omega\)}; &
	; \\

	\node (M) {\(M\)}; &
	\node(Quotient) {\(\q{\Omega}{\sim}\)}; &
	\node (Mp) {\(M'.\)}; \\
};

\draw[->] (Omega) -- (Mp)			 node[anchor=west] {\(\ \Theta'\)};
\draw[<-] (Mp)		 -- (Quotient) node[anchor=north] {\(\tilde\Theta'\)};

\draw[->] (Omega) -- (Quotient)  node[anchor=east] {};

\draw[->] (Omega) -- (M)			 node[anchor=east] {\(\Theta\ \)};
\draw[<-] (M)		 -- (Quotient) node[anchor=north] {\(\tilde\Theta\)};
\end{tikzpicture}\]
Since there is at most one smooth structure on \(\q{\Omega}{\sim}\) which makes the quotient map a surjective submersion, the smooth structure on \(\q{\Omega}{\sim}\) given by the bijection \(\tilde \Theta\) must be diffeomorphic to the one given by the bijection \(\tilde \Theta'\). Hence \(\tilde\Theta' \circ \tilde\Theta^{-1}\) is a diffeomorphism. If the maps \(\Theta\) and \(\Theta'\) are suitably chosen, this construction also allows us to determine additional structure, such as metrics or conformal types of metrics on the manifolds. 
Of course, the equivalence relation \(\sim\) is uniquely determined by \(\Theta\). However, the point is that \(\sim\) should be possible to characterize completely in terms of the data, whereas \(\Theta\) necessarily depends on \(M\).

The method described above will be used explicitly in the proof of the first main result, \prettyref{thm:Fireworks data determines isometry class}. The set \(\Omega\) will consist of vectors and curve lengths, and the map \(\Theta\) will be the exponential map. In \prettyref{thm:Past sky shadow data determines conformal class} we will use the same idea but not explicitly construct an equivalence relation. The set \(\Omega\) will consist of null Weingarten maps (as defined in \prettyref{app:Null geometry}) of sky shadows, and the map \(\Theta\) will be defined by solving the Riccati equation for the null Weingarten map and using the solution to identify the point from which the sky shadow originated.

\section{Fireworks rigidity}
\label{sec:Fireworks rigidity}
We will now define the concept of fireworks data, which should be thought of as a set of triples which represent broken geodesics which are first timelike and then lightlike. 
We use the notation \(\TdM\) to denote the boundary of the tangent bundle of a manifold \(\M\) with boundary; in other words \(\TdM\), which could also be denoted \(\left.T\M\right|_{\partial \M}\), is the restriction of the tangent bundle of \(\M\) to the boundary \(\partial \M\). Note that this is different from the tangent bundle of \(\partial \M\) itself, which is denoted by \(T\partial \M\).
\begin{defn}
Let \(\M\) be a time-oriented Lorentzian manifold with boundary \(\partial \M\). The \emph{fireworks data} of \(\M\), as illustrated in \prettyref{fig:Fireworks data}, is the set of triples \((\xi, t, \eta) \in \TdM \times \R \times \TdM\) such that
\begin{itemize}
\item \(\xi\) is future-directed and timelike,
\item \(\eta\) is future-directed and lightlike,
\item after distance \(t\), the geodesic starting with \(\xi\) intersects the geodesic ending with \(\eta\).
\end{itemize}
The geodesics are allowed to intersect \(\partial \M\) in arbitrarily many points.
We say that \(\M\) and \(\Mprime\) have \emph{isomorphic} fireworks data if \(\TdM \to \partial \M\) and \(\TdMprime \to \partial \Mprime\) are isomorphic as smooth fiber bundles, by an isomorphism which identifies the fireworks data of \(\M\) with the fireworks data of \(\Mprime\)\!.
\end{defn}

In this section and in later sections we will use a function \(T_\partial\) which for a given vector tells us how long, in parameter time, it takes for the geodesic starting with that vector to intersect the boundary. The following lemma tells us that this function is smooth on a certain open submanifold of the tangent bundle.
\begin{lem}
\label{lem:Smooth time to boundary}
Let \(\M\) be a semi-Riemannian manifold with boundary and let \(M\) be its interior. Let \(\oTM\) denote the set of tangent vectors based in \(M\) which are initial velocities of inextendible half-geodesics which intersect the boundary and do so transversely. Define \(T_\partial \colon \oTM \to \R\) by letting \(T_\partial(\eta)\) be the parameter time after which the geodesic starting with \(\eta\) intersects the boundary. (By definition of \(\oTM\), there is precisely one such parameter time.) Then \(\oTM\) is an open submanifold of \(TM\) and \(T_\partial\) is smooth.
\end{lem}
\begin{proof}
First extend \(T_\partial\) to all of \(TM\) by letting \(T_\partial(\eta)\) be the parameter time after which the geodesic starting with \(\eta\) intersects \(\partial \M\) for the \emph{first} time. We let \(T(\eta) = \infty\) if the geodesic does not intersect \(\partial \M\) at all.
Extend \(\M\) to a semi-Riemannian manifold without boundary \(\mcM\) and let \(\exp \colon TM \to \mcM\) denote the (partially defined) exponential map of this larger manifold. 
Fix \(\eta \in \oTM\) and let \(p = \exp(T_\partial(\eta)\eta) \in \partial \M\). Choose a neighborhood \(U \subset \mcM\) of \(p\) foliated by level sets of a smooth submersion \(F \colon U \to \R\) such that \(F^{-1}(0) = U \cap \partial \M\). 
This is possible since \(\partial \M\) is an embedded hypersurface in \(\mcM\).
For \((\xi, t)\) in a neighborhood of \((\eta, T_\partial(\eta)) \in TM \times \R\), let
\[\Phi(\xi, t) = F(\exp(t\xi)).\]
(There is an open set where this is well-defined, since \(\exp\) is continuous so that \(\exp^{-1}(U)\) is open.)
Consider the equation
\[\Phi(\xi, t) = 0.\]
Since \(F\) is a submersion and the geodesic starting with \(\eta\) intersects \(\partial \M\) transversely, we have
\[
\frac{\partial \Phi}{\partial t}(\eta, T_\partial(\eta))
=
\left.\frac{d}{d t}\right|_{t = T_\partial(\eta)} F(\exp(t\eta))
\neq
0.
\]
Hence, by the implicit function theorem \cite[Theorem C.40]{LeeSmoothManifolds} there is a smooth function \(t\) defined on a neighborhood \(V\) of \(\eta\) such that
\[
\Phi(\xi, t(\xi))
=
0
\text{ for } \xi \in V,
\]
\[t(\eta) = T_\partial(\eta).\]
Moreover, this function is locally unique in the sense that if \(\epsilon > 0\) is sufficiently small then for each \(\xi \in V\) there is at most one \(t\) with \(|t - t(\eta)| < \epsilon\) such that \(F(\exp(t\xi)) = 0\).
Fix such an \(\epsilon\). 

We will now prove that, after possibly shrinking \(V\), the function \(T_\partial\) agrees with \(t\) on \(V\). 
Note that
\[T_\partial(\xi) = \inf \{t \geq 0 \colon t\xi \in \exp^{-1}(\partial \M)\}\]
which means that \(T_\partial(\xi) \leq t(\xi)\). Moreover, \(t\) is continuous so we may shrink \(V\) so that \(|t(\xi) - t(\eta)| < \epsilon\) for all  \(\xi \in V\). Hence \(T_\partial(\xi) < t(\eta) + \epsilon\) for all \(\xi \in V\).
Since the set \(\partial \M\) is closed and \(\exp\) is continuous, the set \(\exp^{-1}(\partial \M)\) is also closed and hence \(T_\partial\) is lower semi-continuous. 
This means that we may shrink \(V\) to ensure that \(T_\partial(\xi) > T_\partial(\eta) - \epsilon\) for all \(\xi \in V\). Since \(t(\eta) = T_\partial(\eta)\) we have now obtained the inequalities
\[t(\eta) - \epsilon < T_\partial(\xi) < t(\eta) + \epsilon\]
for all \(\xi \in V\). Note that both \(t = t(\xi)\) and \(t = T_\partial(\xi)\) are such that \(|t - t(\eta)| < \epsilon\) and \(F(\exp(t\xi)) = 0\). There cannot be more than one such value of \(t\) by our choice of \(V\) and \(\epsilon\), and hence \(t(\xi) = T_\partial(\xi)\). This means that \(T_\partial\) agrees with \(t\) on \(V\), so that \(T_\partial\) is smooth on \(V\).

If we can show, possibly after shrinking \(V\) further, that \(V \subseteq \oTM\), then we will have shown that \(\oTM\) is open and that the restriction of \(T_\partial\) to \(\oTM\) is smooth. Let \(\Phi_g\) denote geodesic flow. Explicitly, \(\Phi_g(t, \xi) = \dot\gamma_\xi(t)\) where \(\gamma_\xi\) is the geodesic with \(\dot\gamma_\xi(0) = \xi\). Then the map \(\alpha \colon \xi \mapsto \Phi_g(T_\partial(\xi), \xi)\) defined on \(V\) is smooth. Note that, for \(\xi \in V\), it holds that \(\alpha(\xi) \in T\partial \M\) if and only if \(\xi \notin \oTM\). The inverse image of \(T\partial \M\) under the smooth map \(\alpha\) is closed since \(\partial \M\) is closed, and it does not contain \(\eta\) since \(\eta \in \oTM\). Hence \(V \setminus \alpha^{-1}(T\partial \M)\) is an open neighborhood of \(\eta\) which is completely contained in \(\oTM\). We have shown that \(\oTM\) is open and that \(T_\partial \colon \oTM \to [0, \infty)\) is smooth.
\end{proof}

The proof of \prettyref{thm:Fireworks data determines isometry class} will involve an equivalence relation \(\asymp\). To prove one of the properties of that equivalence relation we will use the following point-set topology result.
\begin{lem}
\label{lem:Topological lemma}
Let \(\Gamma\subseteq\left[0,1\right]\times\R\) be compact. Define an equivalence relation \(\approx\) on \(\Gamma\) by letting \(\left(s,t\right)\approx\left(s',t'\right)\) if and only of \(s = s'\). Suppose that the restriction to \(\Gamma\) of the projection \(\pi_{1}\colon\left[0,1\right]\times\R\to\left[0,1\right]\) is surjective. Then the quotient space \(\q{\Gamma}{\approx}\) is connected.
\end{lem}
\begin{proof}
The quotient map \(\Gamma \to \q{\Gamma}{\approx}\) is continuous, so \(\q{\Gamma}{\approx}\) is compact.
Then \(\pi_1 \colon \Gamma \to [0, 1]\) descends to a continuous bijection from the compact space \(\q{\Gamma}{\approx}\) to the Hausdorff space \([0, 1]\), so it is a homeomorphism.
Hence \(\q{\Gamma}{\approx}\) is homeomorphic to \([0, 1]\), and in particular connected.
\end{proof}

We are now ready to prove that the fireworks data of a compact time-oriented Lorentzian manifold with boundary determines the manifold uniquely when the strong causality condition and a weak convexity condition are imposed. The convexity condition is the condition that each point in the interior of the manifold is the future endpoint of a past-inextendible timelike geodesic which intersects the boundary transversely. The following sketch of a construction illustrates which kind of manifolds may be excluded by this condition.
Consider the outside of a coordinate sphere \(t^2 + x_1^2 + \cdots + x_n^2 = r^2\) in the Minkowski spacetime with standard coordinates \((t, x_1, \ldots, x_n)\), and consider a family of future-directed geodesics which are initially tangent to this sphere and whose initial velocities lie in the planes spanned by the coordinate radial direction and the \(\partial_t\) direction. They will initially look like in the left picture in \prettyref{fig:Counterexample with tangent geodesics}. After a parameter time which is long compared to the size of the removed ball the geodesics will look like in the right picture in \prettyref{fig:Counterexample with tangent geodesics}: They are very close to originating from a single point. It seems plausible that the metric can be perturbed slightly to make the geodesics appear to have originated exactly from a single point. There are Lorentzian manifolds (for instance the Anti-de Sitter spacetime) where all timelike geodesics from one point refocus in another point. Using a similar metric, it seems plausible that one can make the family of geodesics we are considering refocus in a single point \(p\). This means that all past-directed timelike geodesics starting at \(p\) intersect the boundary tangentially, and the manifold fails to satisfy the hypothesis of the theorem.
Nevertheless, we do believe that the theorem is true even without this hypothesis since the set of points where it fails is small.

\begin{figure}
	\centering{}
	\includegraphics[width=0.45\textwidth]{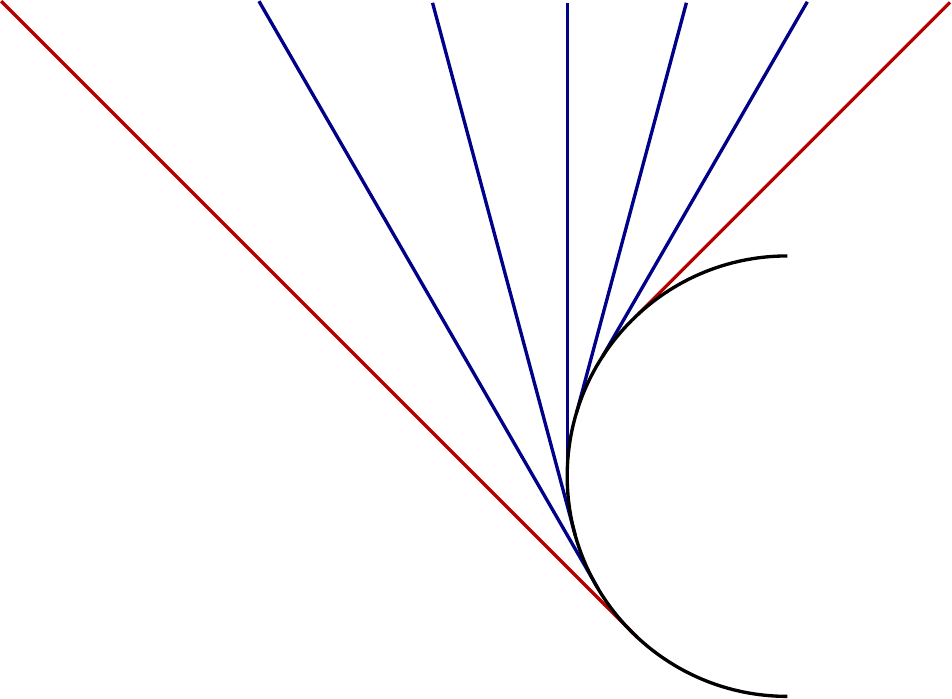}
	\includegraphics[width=0.45\textwidth]{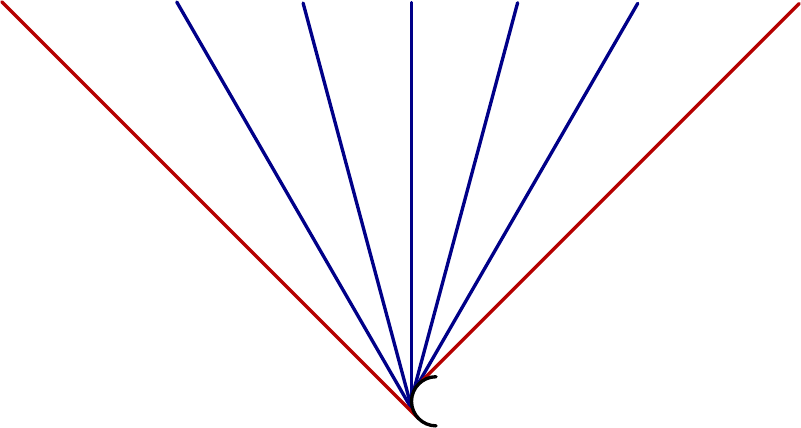}
	\caption{Geodesics tangent to a small sphere in the \((1+1)\)-dimensional Minkowski plane.}
\label{fig:Counterexample with tangent geodesics}
\end{figure}

\begin{thm}
\label{thm:Fireworks data determines isometry class}
Let \((\M_1, g_1)\) and \((\M_2, g_2)\) be strongly causal, compact, time-oriented Lorentzian manifolds of dimension \(n \geq 3\) with boundary. Suppose that each point \(p \in M_i\) in the interior \(M_i\) of \(\M_i\) is the future endpoint of a past-inextendible timelike geodesic which intersects \(\partial \M_i\) transversely. 
If \(\M_1\) and \(\M_2\) have isomorphic fireworks data, then they are isometric. 

In other words, the isometry type of a strongly causal, compact, time-oriented Lorentzian manifold of dimension \(n \geq 3\) with boundary with the above transversality condition for geodesics is uniquely determined by fireworks data.
\end{thm}

The proof of \prettyref{thm:Fireworks data determines isometry class} can be extended to cover some settings in which the manifold is not necessarily compact and the data is not necessarily given at the boundary. This gives rise to further complications, in particular when the manifold is noncompact and the set where the data is given is not connected. For this reason we will consider only the case of compact manifolds with boundary.

We have drawn two-dimensional illustrations for the proof, even though the case \(n = 2\) is not covered by the theorem.

\begin{proof}
The proof follows the general idea described in \prettyref{sec:Reconstruction of smooth structures}, and consists of eight steps:
\begin{enumerate}
\item\label{step:Choose Omega} Consider a smooth fiber bundle \(E \to B\) and a subset \(D \subseteq E \times [0, \infty) \times E\), and construct a certain open submanifold \(\Omega \subseteq E \times (0, \infty)\). At this point \(E\) and \(D\) may be arbitrary, but in later steps \(E\) will be \(\TdM\) and \(D\) will be fireworks data for a manifold \(\M\). The set \(\Omega\) will consist of vectors in \(\TdM\) and curve lengths, which should be thought of as representing geodesic segments.
\item\label{step:Construct equivalence relation} Construct an equivalence relation \(\sim\) on \(\Omega\).
\item\label{step:Note diffeomorphism} Note that if there is a topology and a smooth structure on the quotient \(\q{\Omega}{\sim}\) such that the quotient map \(p \colon \Omega \to \q{\Omega}{\sim}\) is a surjective submersion, then this topology and smooth structure is unique.
\item\label{step:Note metric} Note that if there is a time-oriented Lorentzian metric on the quotient such that the tangents of the curves \(t \mapsto p(\xi, t)\) form a subset with fiberwise nonempty interior of the bundle of future-directed unit timelike tangent vectors based in \(\q{\Omega}{\sim}\) (in the sense that it intersects each fiber in a subset with nonempty interior) then this metric is unique.
\item\label{step:Choose submersion} Suppose that a time-oriented Lorentzian manifold \((\M, g)\) with interior \(M\) is such that \(\TdM\) is isomorphic as a smooth fiber bundle to \(E\) by an isomorphism which identifies the fireworks data with \(D\). Choose a surjective submersion \(\Theta \colon \Omega \to M\) and a bijection \(\tilde\Theta \colon \q{\Omega}{\sim} \to M\) such that the following diagram commutes:\\
\[\begin{tikzpicture}[every node/.style={midway}]
\matrix[column sep={5em,between origins}, row sep={2em}] at (0,0) {
	; & \node(Omega) {\(\Omega\)}; \\

	\node (M) {\(M\)}; &
	\node(Quotient) {\(\q{\Omega}{\sim}.\)};\\
};

\draw[->] (Omega) -- (M)			node[anchor=south] {\(\Theta\ \)};
\draw[->] (Quotient)	-- (M) 	node[anchor=north] {\(\tilde\Theta\)};
\draw[->] (Omega) -- (Quotient) node[anchor=west] {\(p\)};
\end{tikzpicture}\]
This submersion \(\Theta\) will essentially be the exponential map.
\item\label{step:Identify timelike unit vectors} Show that this choice of \(\Theta\) is such that the tangents of the curves \(t \mapsto \Theta(\xi, t)\) form a subset with fiberwise nonempty interior of the bundle of future-directed unit timelike tangent vectors based in \(M\). 
\item\label{step:Conclude isometry of the interior} From this, conclude that the topology, smooth structure, and metric on \(M\) is uniquely determined.
\item\label{step:Conclude isometry including boundary} Apply~\cite[Theorem 5.3]{Chrusciel10} to determine the metric on all of \(\M\).
\end{enumerate}

\begin{figure}
	\centering{}
	\begin{overpic}[height=0.3\textheight]{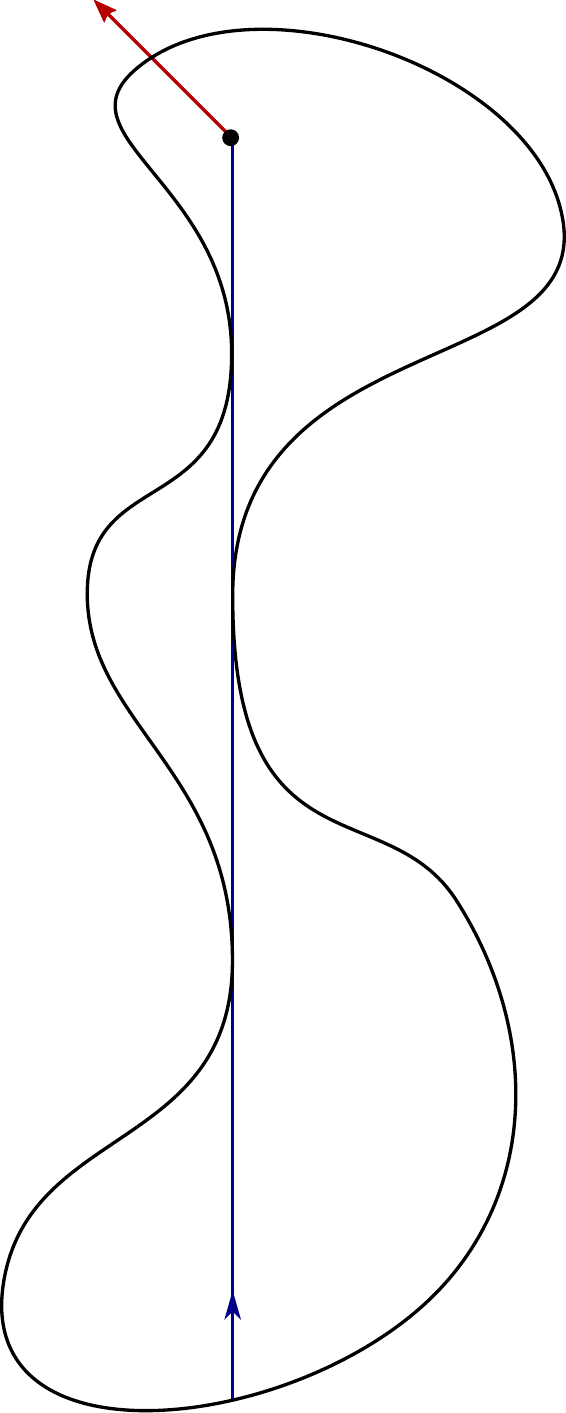}
	 \put (18,87) {\(\exp(t\xi)\)}
	 \put (30,69) {\(\partial \M\)}
	 \put (18,10) {\(\xi\)}
	\end{overpic}
	\caption{A Lorentzian manifold consisting of a compact subset of the Minkowski plane, where \(\widehat \Omega\) is not a manifold. The same phenomenon occurs in higher dimensions.}
\label{fig:Non-manifold Omega}
\end{figure}

\stepheader{Step~\ref{step:Choose Omega}}
Let \(\widehat \Omega\) be the image of \(D\) under the projection \(\pi \colon E \times \R \times E \to E \times \R\) onto the first two components,
\[\widehat \Omega = \pi(D).\]
This set \(\widehat \Omega\) is not necessarily a manifold, because of the phenomenon shown in \prettyref{fig:Non-manifold Omega}, and for some of the arguments we will need a manifold. However, the interior of \(\widehat\Omega\) is an open subset of \(E \times \R\) and so inherits a smooth manifold structure. Let \(\Omega\) be this interior:
\[\Omega = \interior \pi(D).\]
The manifold \(\Omega\) will be the object of main interest, but we will need \(\widehat \Omega\) to construct an equivalence relation \(\sim\) on \(\Omega\).

We will now say a few words about what properties the sets \(\Omega\) and \(\widehat \Omega\) have in the case when \(D\) is fireworks data of an actual strongly causal, compact, time-oriented Lorentzian manifold \(\M\) with boundary \(\partial \M\), and \(E = \TdM\). In this case, \(\Omega\) consists precisely of the pairs \((\xi, t)\) (where \(t > 0\) and \(\xi\) is timelike, future-directed, and based at \(\partial \M\)) such that \(\xi\) is transverse to \(\partial \M\) and the geodesic of length \(t\) starting with \(\xi\) does not intersect \(\partial \M\) except at its initial endpoint. 
Similarly, every element \((\xi, t) \in \widehat \Omega\) is such that the geodesic starting with \(\xi\) has length at least \(t\).

\begin{figure}
 \centering{}
 \begin{overpic}[width=0.3\textwidth]{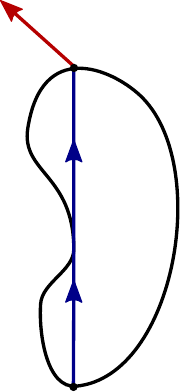}
  \put (21,38) {\(q\)}
  \put (20,87) {\(\substack{\begin{aligned}r&= \exp(\mathfrak t\zeta)\\[-.5em] &= \exp((\mathfrak s + \mathfrak t)\xi)\end{aligned}}\)}
  \put (-3,95) {\(\eta\)}
  \put (22,25) {\(\xi\)}
  \put (22,60) {\(\zeta\)}
 \end{overpic}
 \caption{The equivalence relation \(\asymp\) will identify \((\xi, \mathfrak s + \mathfrak t)\) and \((\zeta, \mathfrak t)\) where \(\mathfrak s\) is such that \(\exp(\mathfrak s\xi) = q\) and \(\mathfrak t > 0\) is arbitrary (though small enough for the geodesics to stay in the manifold.)}
\label{fig:Auxiliary equivalence relation}
\end{figure}

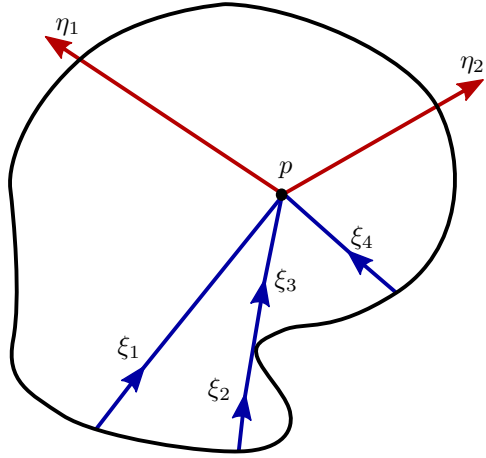
\begin{figure}
 \centering{}
 \begin{overpic}[width=0.5\textwidth]{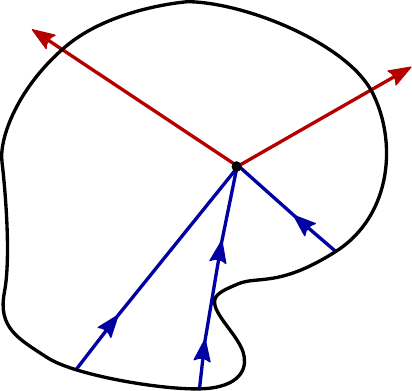}
  \put (57,59) {\(p\)}
  \put (23,21) {\(\xi_1\)}
  \put (42,12) {\(\xi_2\)}
  \put (56,35) {\(\xi_3\)}
  \put (72,44) {\(\xi_4\)}
  \put (10,89) {\(\eta_1\)}
  \put (95,81) {\(\eta_2\)}
 \end{overpic}
 \caption{Let \(t_i\) be such that \(\exp(t_i\xi_i) = p\). Let \(\Sigma = \{\lambda \eta_1, \lambda \eta_2 \mid \lambda \in \R^+\}\). Then \((\xi_1, t_1) \sim (\xi_2, t_2)\) since there is a path in \(\widehat \Omega_\Sigma\) connecting them. For the same reason \((\xi_3, t_3) \sim (\xi_4, t_4)\). Finally \((\xi_2, t_2) \sim (\xi_3, t_3)\) since they are identified by \(\asymp\). Hence the equivalence relation \(\sim\) captures the fact that \(\exp(t_1\xi_1) = \exp(t_4\xi_4)\).}
\label{fig:Equivalence relation}
\end{figure}

\stepheader{Step~\ref{step:Construct equivalence relation}}
The set \(\widehat \Omega\) can be partitioned into sets \(\widehat \Omega_\Sigma\) indexed by nonempty subsets \(\Sigma \subseteq E\), by letting
\[\widehat \Omega_\Sigma = \left\{(\xi, t) \in E \times [0, \infty) \mid (\xi, t, \eta) \in D \iff \eta \in \Sigma\right\}.\]
In other words, \(\widehat \Omega_\Sigma\) consists of those pairs \((\xi, t)\) such that the inverse image \(\pi^{-1}(\xi, t)\) is precisely \(\{(\xi, t)\} \times \Sigma\).
Of course, \(\widehat \Omega_\Sigma\) will be empty for most choices of \(\Sigma\). 
Equip each set \(\widehat \Omega_\Sigma\) with the subspace topology. Let \(\sim\) be the equivalence relation on \(\Omega\) defined as follows:
\begin{itemize}
\item Define an auxiliary equivalence relation \(\asymp\) on \(\widehat \Omega\) by letting \(\left(\xi,t\right) \asymp \left(\zeta,s\right)\) if for every \(\eta \in E\) and every \(\tau \in [0, \infty)\) it holds that
\[\left(\xi, t + \tau, \eta\right) \in D \iff \left(\zeta, s + \tau, \eta\right) \in D.\]
See \prettyref{fig:Auxiliary equivalence relation} for an illustration.
\item Define \(\sim\) on \(\Omega\) by letting \(\left(\xi, t\right) \sim \left(\zeta, s\right)\) if there is some \(\Sigma \subseteq E\) such that \((\xi,t)\) and \((\zeta,s)\) both belong to \(\widehat \Omega_\Sigma\) and such that they belong to the same connected component of \(\q{\widehat \Omega_\Sigma}{\asymp}\). 
See \prettyref{fig:Equivalence relation} for an illustration.
(Here \(\q{\widehat \Omega_\Sigma}{\asymp}\) is equipped with the quotient topology.)
\end{itemize}

\stepheader{Step~\ref{step:Note diffeomorphism}}
Suppose now that \(\mathcal A\) and \(\mathcal A'\) are smooth structures (including topologies) on the set \(\q{\Omega}{\sim}\) making it into a smooth manifold such that the quotient map \(p \colon \Omega \to \q{\Omega}{\sim}\) is a smooth submersion with respect to both of them. Then 
the smooth manifolds \(\left(\q{\Omega}{\sim}, \mathcal A\right)\) and \(\left(\q{\Omega}{\sim}, \mathcal A'\right)\) are diffeomorphic. 
This can be seen as follows: Choose a point \({q \in \q{\Omega}{\sim}}\). Since \(p \colon \Omega \to \left(\q{\Omega}{\sim}, \mathcal A\right)\) is a surjective submersion, it has a smooth section \(s \colon (V, \mathcal A) \to \Omega\) on an open \(\mathcal A\)-neighborhood \(V\) of \(q\), and we get the following commutative diagram:
\[\begin{tikzpicture}[every node/.style={midway}]
\matrix[column sep={3.5em,between origins}, row sep={2.5em}] at (0,0) {
	; & ; & \node(Omega) {\(\Omega\)}; & ; & ; \\

	\node(V) {\quad\(V\)}; & ; &
	\node(Quotient1) {\(\left(\q{\Omega}{\sim}, \mathcal A\right)\)}; &
		; &
	\node(Quotient2) {\(\left(\q{\Omega}{\sim}, \mathcal A'\right).\)}; \\
};

\draw[right hook->] (V)			-- (Quotient1) node[anchor=south] {};
\draw[->] (V)			-- (Omega) node[anchor=south] {\(s\ \)};
\draw[->] (Omega)			-- (Quotient1) node[anchor=west] {\(p\)};
\draw[->] (Omega)			-- (Quotient2) node[anchor=south] {\(\ \ p'\)};
\draw[<->] (Quotient1) -- (Quotient2) node[anchor=north] {\(\id\)};
\end{tikzpicture}\]
This means that \(p' \circ s \colon (V, \mathcal A) \to \left(\q{\Omega}{\sim}, \mathcal A'\right)\) is smooth. The map \(p' \circ s\) is the restriction of the identity map \(\id \colon \left(\q{\Omega}{\sim}, \mathcal A\right) \to \left(\q{\Omega}{\sim}, \mathcal A'\right)\) to the neighborhood \(V\) of \(q\). Since \(q\) was arbitrary, this shows that the identity is a smooth map. Repeating the argument with \(\mathcal A\) and \(\mathcal A'\) reversed shows that its inverse is also smooth. Hence it is a diffeomorphism.

Equip \(\q{\Omega}{\sim}\) with this unique topology and smooth structure, whenever it exists.

\stepheader{Step~\ref{step:Note metric}}
Suppose that \(g_1\) and \(g_2\) are time-oriented Lorentzian metrics on \(\q{\Omega}{\sim}\) such that the tangents of the curves \(t \mapsto p(\xi, t)\) form a subset with fiberwise nonempty interior of the bundle of future-directed unit timelike tangent vectors in \(g_i\). 
Then the tangent map of the identity \(\id \colon \left(\q{\Omega}{\sim}, g_1\right) \to \left(\q{\Omega}{\sim}, g_2\right)\) maps a nonempty open set of future-directed unit timelike vectors at each point to future-directed unit timelike vectors, and hence the identity \(\id\) is an isometry. Equip \(\q{\Omega}{\sim}\) with this unique metric, whenever it exists.

\stepheader{Step~\ref{step:Choose submersion}}
Let \((\M, g)\) be a time-oriented Lorentzian manifold with boundary satisfying the hypotheses of the theorem: 
\begin{itemize}
\item \(\M\) is strongly causal and compact.
\item Every point in the interior \(M\) of \(\M\) is the future endpoint of a past-inextendible timelike geodesic intersecting the boundary transversely.
\end{itemize}
Suppose moreover that \(\TdM \cong E\) by a fiber bundle isomorphism which identifies the fireworks data of \(\M\) with \(D\).
Our first task is to construct a surjective submersion \(\Theta \colon \Omega \to M\), and for this we will use the exponential map of \(\M\). Let \(n\) be the normalization map defined by
\[n(\xi) = \frac{\xi}{\sqrt{-g(\xi, \xi)}}\]
for timelike \(\xi\). Define
\[\Theta(\xi, t) = \exp\left(t n(\xi)\right).\]
We first need to show that this definition makes sense. Note that \(\Omega\) contains only timelike vectors \(\xi\), so \(n(\xi)\) is well-defined. It is also necessary that there is a geodesic segment in \(\M\) of length \(t\) starting with \(\xi\) and that this segment ends in the interior of \(\M\). The first part is guaranteed by definition of \(\Omega\), since this definition tells us that \((\xi, t, \eta) \in D\) for some lightlike \(\eta \in \TdM\), which in turn means that there is a broken geodesic starting with \(\xi\) the timelike part of which exists for a distance \(t\). To see that \(\Theta\left(t n(\xi)\right) \notin \partial \M\), note that if this were not the case then \((\xi, t)\) would not be an interior point of \(\Omega \subseteq E \times [0, \infty)\), contradicting the openness of \(\Omega\). This shows that \(\Theta\) is well-defined.

To show that \(\Theta \colon \Omega \to M\) is a surjective submersion, we will construct smooth local sections  around every point of \(M\). Choose a point \(p \in M\). By hypothesis, there is at least one timelike geodesic with future endpoint \(p\) which intersects \(\partial \M\) transversely. Let \(X(p)\) be a future-directed tangent vector to any such geodesic at \(p\). Extend \(X(p)\) to a smooth vector field \(X\) on a neighborhood \(U\) of \(p\). By shrinking \(U\) we can ensure (using \prettyref{lem:Smooth time to boundary}) that the geodesics ending with \(X(q)\) for \(q \in U\) intersect \(\partial \M\) transversely. Define \(s \colon U \to \Omega\) by
\[
s(q) = (\Phi_g(-T_\partial(-X(q)), X(q)), T_\partial(-X(q))) \in \Omega
\]
where \(\Phi_g\) is the geodesic flow and \(T_\partial(-X(q))\) is the time after which the geodesic starting with \(-X(q)\) hits the boundary of \(\M\). (Note that \(s\) takes values in \(\Omega\) since the strong causality condition and compactness of \(\M\) tells us that each past-directed lightlike geodesic from each point in \(M\) must intersect \(\partial \M\).) The function \(T_\partial\) is smooth by \prettyref{lem:Smooth time to boundary}, so \(s\) is also smooth. Now \(\Theta(s(q)) = q\) for all \(q \in U\) so \(s\) is a section of \(\Theta\) on the neighborhood \(U\) of the arbitrary point \(p\). By doing this for all timelike geodesics with future endpoint \(p\) which intersect \(\partial \M\) transversely, we get local sections which hit all preimages of \(p\) under \(\Theta\). Hence \(\Theta\) is a surjective submersion.

Extend \(\Theta \colon \Omega \to M\) to a continuous map \(\widehat \Theta \colon \widehat \Omega \to \M\) using the same formula as for defining \(\Theta\). This is well-defined since the geodesic starting with \(\xi\) has length at least \(t\) if \((\xi, t) \in \widehat \Omega\). (Note that it does not make sense to ask whether \(\widehat \Theta\) is smooth, since we have not given \(\widehat \Omega\) a smooth structure.)

We will show that \(\Theta\) descends to a bijection \(\q{\Omega}{\sim} \to M\), in other words that there is a bijection \(\tilde \Theta\) such that the following diagram commutes:\\
\[\begin{tikzpicture}[every node/.style={midway}]
\matrix[column sep={5em,between origins}, row sep={2em}] at (0,0) {
	; & \node(Omega) {\(\Omega\)}; \\

	\node (M) {\(M\)}; &
	\node(Quotient) {\(\q{\Omega}{\sim}\).};\\
};

\draw[->] (Omega) -- (M)			node[anchor=south] {\(\Theta\ \)};
\draw[->] (Quotient)	-- (M) node[anchor=north] {\(\tilde\Theta\)};
\draw[->] (Omega) -- (Quotient) node[anchor=west] {\(p\)};
\end{tikzpicture}\]
To show this, it is sufficient to show that \(\Theta(\xi_0, t_0) = \Theta(\xi_1, t_1)\) if and only if \({(\xi_0, t_0) \sim (\xi_1, t_1)}\). Suppose first that \((\xi_0,t_0) \sim (\xi_1, t_1)\). By definition, this means that \((\xi_0, t_0)\) and \((\xi_1, t_1)\) belong to the same connected component \(C\) of \(\q{\widehat \Omega_\Sigma}{\asymp}\) for some \(\Sigma \subseteq E\). We will show that \(\widehat \Theta\) descends to a map defined on the quotient \(\q{\widehat \Omega_\Sigma}{\asymp}\), and that this map maps all of \(C\) to a single point, so that in particular \(\Theta(\xi_0, t_0) = \Theta(\xi_1, t_1)\). For both these statements, we will need the following fact: Every point of \(\M\) (including boundary points) has a neighborhood \(V\) such that if \(\Sigma \subseteq E\) contains two nonparallel vectors then \(\widehat \Theta(\widehat \Omega_\Sigma) \cap V\) contains at most one element. To see this, take a point \(q \in \M\). If \(\Sigma\) contains a non-lightlike vector, then \(\widehat \Omega_\Sigma = \emptyset\), so we may assume that \(\Sigma\) contains only lightlike vectors. Temporarily enlarge \(\M\) to a Lorentzian manifold \(\mathcal M\) without boundary. Let \(U \subseteq \mathcal M\) be a convex neighborhood (as defined in \prettyref{app:Assorted definitions}) of \(q\) in \(\mathcal M\). Since \(\M\) is strongly causal and \(U \cap \M\) is an open (in \(\M\)) neighborhood of \(q\), there is a causally convex (as defined in \prettyref{app:Assorted definitions}) open neighborhood \(V \subseteq U \cap \M\) of \(q\). Suppose that \(\widehat \Theta(\widehat \Omega_\Sigma) \cap V\) has two distinct elements \(q \neq \hat{q}\). Then the family of lightlike geodesics in \(\M\) from \(\Sigma\) passes through both \(q\) and \(\hat{q}\). The geodesic segments between these points must be completely contained in \(V\) since \(V\) is causally convex. Since \(\Sigma\) contains two nonparallel vectors, this means that we have more than one geodesic segment in \(U\) connecting \(q\) and \(\hat{q}\), which contradicts \(U\) being a convex neighborhood. Hence \(\widehat \Theta(\widehat \Omega_\Sigma) \cap V\) can have at most one element.

{
\newcommand{\tf}{\mathfrak t}
To see that \(\Theta\) descends to a map on the quotient \(\q{\widehat \Omega}{\asymp}\), suppose \((\zeta_0, s_0) \asymp (\zeta_1, s_1)\), or equivalently that for all \(\tau \in [0, \infty)\) and \(\eta \in E\) it holds that
\[
(\zeta_0, s_0 + \tau, \eta) \in D \iff (\zeta_1, s_1 + \tau, \eta) \in D.
\]
Note that the following statements are equivalent:
\begin{itemize}
\item \((\zeta_0, s_0 + \tau)\) belongs to the domain of \(\widehat\Theta\).
\item There exists some \(\eta \in E\) such that \((\zeta_0, s_0 + \tau, \eta) \in D\). 
\item There exists some \(\eta \in E\) such that \((\zeta_1, s_1 + \tau, \eta) \in D\). 
\item \((\zeta_1, s_1 + \tau)\) belongs to the domain of \(\widehat\Theta\).
\end{itemize}
The set of \(\tau \in [0, \infty)\) such that \(\widehat\Theta(\zeta_0, s_0 + \tau)\) is defined (or equivalently such that \(\widehat\Theta(\zeta_1, s_1 + \tau)\) is defined) is closed and bounded, and so has a maximum. Let \(\tf\) denote this maximum. Our first claim is that \({\widehat \Theta(\zeta_0, s_0 + \tf) = \widehat \Theta(\zeta_1, s_1 + \tf)}\). To see this, note that \(\widehat \Theta(\zeta_0, s_0 + \tf) \in \partial \M\), since otherwise the geodesic starting with \(\zeta_0\) would be extendible beyond time \(s_0 + \tf\). Similarly, \({\widehat \Theta(\zeta_1, s_1 + \tf) \in \partial \M}\). 
Choose \(\eta \in T_{\widehat\Theta(\zeta_1, s_1 + \tf)}\M\) outward-directed and future-directed such that \((\zeta_1, s_1 + \tf, \eta) \in D\). Then \((\zeta_0, s_0 + \tf, \eta) \in D\). If it were the case that \({\widehat \Theta(\zeta_0, s_0 + \tf) \neq \widehat \Theta(\zeta_1, s_1 + \tf)}\), then the inextendible geodesic ending with \(\eta\) must pass through \(\widehat \Theta(\zeta_0, s_0)\) and \(\widehat \Theta(\zeta_1, s_1)\) in that order. Reversing the roles of \((\zeta_1, s_1)\) and \((\zeta_0, s_0)\) we get a lightlike geodesic through these points in the opposite order. Concatenating these geodesic segments, we get a closed piecewise smooth future-directed lightlike curve, contradicting the assumption that \(\M\) is strongly causal. Hence it must hold that \({\widehat\Theta(\zeta_0, s_0 + \tf) = \widehat\Theta(\zeta_1, s_1 + \tf)}\).
We will now show that the curves \(\tau \mapsto (\zeta_0, s_0 + \tau, \eta)\) and \(\tau \mapsto (\zeta_1, s_1 + \tau, \eta)\) coincide near \(\tau = \tf\). Let \(V\) be a neighborhood of the common point \({\widehat \Theta(\zeta_0, s_0 + \tf) = \widehat \Theta(\zeta_1, s_1 + \tf)}\) with the property described above: For any \(\mathfrak S \subseteq E\) containing two nonparallel vectors, \(\widehat\Theta(\widehat \Omega_{\mathfrak S}) \cap V\) has at most one element. For all sufficiently small \(\epsilon > 0\) both \(\widehat\Theta(\zeta_0, s_0 + \tf - \epsilon)\) and \(\widehat\Theta(\zeta_1, s_1 + \tf - \epsilon)\) are contained in \(V\). For fixed \(\epsilon\), make the choice
\[{\mathfrak S} = \left\{ \eta \in E \mid (\zeta_0, s_0 + \tf - \epsilon, \eta) \in D \right\} = \left\{ \eta \in E \mid (\zeta_1, s_1 + \tf - \epsilon, \eta) \in D \right\}\]
where the second equality is due to the assumption that \(\left(\zeta_0, s_0\right) \asymp \left(\zeta_1, s_1\right)\). Note that \({\mathfrak S}\) contains many nonparallel elements since every future-directed lightlike geodesic from \(\widehat\Theta(\zeta_0, s_0 + \tf - \epsilon)\) hits \(\partial \M\) at some point by the assumption of strong causality, and no two such geodesics can end in parallel vectors at \(\partial \M\) unless they were parallel to begin with. By definition of fireworks data, these geodesics must be represented in \(D\). Now \(\left(\zeta_0, s_0 + \tf - \epsilon\right)\) and \(\left(\zeta_1, s_1 + \tf - \epsilon\right)\) are both contained in \(\widehat \Omega_{\mathfrak S}\). Since \(\Theta(\widehat \Omega_{\mathfrak S}) \cap V\) has at most one element, we have then shown that \(\widehat\Theta(\zeta_0, s_0 + \tf - \epsilon) = \widehat \Theta(\zeta_1, s_1 + \tf - \epsilon)\) for all sufficiently small \(\epsilon \geq 0\). Since both curves \(\tau \mapsto \widehat \Theta(\zeta_i, s_i + \tau)\) are geodesics, they must agree on their common domain. In particular, they must agree at \(\tau = 0\) so it must hold that \(\widehat\Theta(\zeta_0, s_0) = \widehat\Theta(\zeta_1, s_1)\). This shows that \(\widehat\Theta\) descends to a map on the quotient \(\q{\widehat \Omega}{\asymp}\). Hence it also descends to a map on the quotient \(\q{\widehat \Omega_{\Sigma}}{\asymp}\). This map is continuous since \(\q{\widehat \Omega_{\Sigma}}{\asymp}\) has been given the quotient topology. 
}

We now continue the proof that \(\Theta(\xi_0, t_0) = \Theta(\xi_1, t_1)\) if \((\xi_0,t_0)\sim(\xi_1,t_1)\). The reader is encouraged to keep \prettyref{fig:Equivalence relation} in mind. Suppose first that \((\xi_0, t_0)\) and \((\xi_1, t_1)\) belong to the same connected component \(C\) of the space \(\q{\widehat \Omega_\Sigma}{\asymp}\). The set \(\Sigma\) contains two nonparallel vectors (and in fact infinitely many pairwise nonparallel vectors) since \(\widehat \Omega_\Sigma\) is nonempty. This is because each future-directed lightlike geodesic from each point in \(M\) reaches \(\partial \M\) by compactness and strong causality. We will show that \(\widehat\Theta\) maps all of \(C\) to a single point, so that in particular \(\Theta(\xi_0, t_0) = \Theta(\xi_1, t_1)\), by showing that the set \(\widehat\Theta(\widehat \Omega_\Sigma)\) is discrete. This set consists of points \(q\in M\) such that the past-inextendible lightlike geodesics ending with vectors in \(\Sigma\) all intersect in \(q\). Choose a neighborhood \(V\) of \(q \in \widehat\Theta(\widehat \Omega_\Sigma)\) such that \(\widehat \Theta(\widehat \Omega_\Sigma)\cap V\) consists of at most one point, as described above. This single point must then be \(q\) itself. Since \(q\) was arbitrary, \(\widehat \Theta(\widehat \Omega_\Sigma)\) is a discrete space. The map \(\widehat \Theta \colon \q{\widehat \Omega_\Sigma}{\asymp} \to \widehat \Theta(\widehat \Omega_\Sigma)\) is continuous, and hence maps the connected component \(C\) to a connected component of \(\widehat \Theta(\widehat \Omega_\Sigma)\), which must be a single point since this space is discrete. Hence \(\Theta(\xi_0, t_0)=\Theta(\xi_1, t_1)\) as claimed.

Conversely, suppose that \(\Theta(\xi_0, t_0) = \Theta(\xi_1, t_1)\). In other words, suppose that the geodesics \(\gamma_0\) and \(\gamma_1\) starting with \(\xi_0\) and \(\xi_1\) and parameterized by curve length hit the same point \(q\) after time \(t_0\) and \(t_1\) respectively:
\[\gamma_0(t_0) = q = \gamma_1(t_1).\]
Since \(\Theta(\lambda \xi_i, t_i) = \Theta(\xi_i, t_i)\) and \((\lambda \xi_i, t_i) \asymp (\xi_i, t_i)\) for all \(\lambda > 0\), we may assume without loss of generality that the \(\xi_i\) are unit vectors.
Let \(\zeta_i = -\dot\gamma_i(t_i)\) and let \(\zeta \colon[0,1] \to T_qM\) be a curve of timelike unit vectors at \(q\) with \(\zeta(0) = \zeta_0\) and \(\zeta(1) = \zeta_1\). For \(s \in [0,1]\) let \(\alpha_s(\,\cdot\,)\) denote the past-inextendible geodesic starting with \(\zeta(s)\). Let \(\Gamma=\left\{ \left(s,\tau\right)\in\left[0,1\right]\times\R\mid\alpha_{s}\left(\tau\right)\in\partial \M\right\} \). The set \(\Gamma\) is closed since it is the inverse image of the closed set \(\partial \M\) under the continuous map \(\left(s,\tau\right)\mapsto\exp\left(\tau\zeta\left(s\right)\right)\). Define the equivalence relation \(\approx\) on \(\Gamma\) by letting \(\left(s,t\right)\approx\left(s,t'\right)\) for all \(s\), \(t\), and \(t'\). Note that \(\approx\) is a \enquote{two-dimensional} version of \(\asymp\) in the sense that \((-\dot\alpha_s(t), t) \asymp (-\dot\alpha_{s'}(t'), t')\) if and only if \((s, t) \approx (s', t')\). (For the \enquote{only if} direction, the observation that \(\widehat\Theta\) descends to a map on \(\q{\widehat\Omega}{\asymp}\) is needed.) By definition of \(\sim\), it holds that \((\xi_{0},t_{0})\sim(\xi_{1},t_{1})\) if \(\left(0,t_{0}\right)\) and \(\left(1,t_{1}\right)\) belong to the same connected component of \(\q{\Gamma}{\approx}\), which is true since \prettyref{lem:Topological lemma} tells us that the space \(\q{\Gamma}{\approx}\) is connected. This completes the proof that \(\Theta\left(\xi_{0},t_{0}\right)=\Theta\left(\xi_{1},t_{1}\right)\) if and only if \((\xi_{0},t_{0})\sim(\xi_{1},t_{1})\).

\stepheader{Step~\ref{step:Identify timelike unit vectors}}
We now wish to show that the tangents of the curves \(t\mapsto\Theta\left(\xi,t\right)\) are future-directed unit timelike vectors based in \(M\), and that the unit tangent space at each point contains an open set of such tangent vectors. Each such tangent is future-directed unit and timelike since the curve \(t \mapsto \Theta\left(\xi, t\right)\) is a timelike unit-speed future-directed geodesic. Now consider a point \(p \in M\). By hypothesis there is a future-directed unit timelike geodesic ending with some \(\zeta_0 \in T_pM\) which intersects the boundary transversely. By \prettyref{lem:Smooth time to boundary} there is a neighborhood of \(\zeta_0\) where all vectors are final velocities of geodesics transverse to the boundary. This gives us an open subset \(V\) of the unit timelike tangent space at \(p\). For \(\zeta \in V\), let \(\xi(\zeta) \in \TdM\) denote the initial velocity of the geodesic, and let \(\mathfrak t(\zeta)\) denote its length. Then \(\zeta\) is the tangent of the curve \(t \mapsto \Theta(\xi(\zeta), t)\) at \(\mathfrak t(\zeta)\).

\stepheader{Step~\ref{step:Conclude isometry of the interior}}
Finally, suppose that \((\M_1, g_1)\) and \((\M_2, g_2)\) are two Lorentzian manifolds satisfying the hypotheses of the theorem. Consider the following commutative diagram:\\
\[\begin{tikzpicture}[every node/.style={midway}]
\matrix[column sep={5em,between origins}, row sep={2em}] at (0,0) {
	; &
	\node(Omega) {\(\Omega\)}; &
	; \\

	\node (M1) {\(M_1\)}; &
	\node(Quotient) {\(\q{\Omega}{\sim}\)}; &
	\node (M2) {\(M_2\)}; \\
};

\draw[->] (Omega) -- (M2)			 node[anchor=west] {\(\Theta_2\)};
\draw[->] (Quotient)	-- (M2) node[anchor=north] {\(\tilde\Theta_2\)};

\draw[->] (Omega) -- (Quotient)  node[anchor=east] {\(p\)};

\draw[->] (Omega) -- (M1)			 node[anchor=east] {\(\Theta_1\ \)};
\draw[->] (Quotient)	-- (M1) node[anchor=north] {\(\tilde\Theta_1\)};
\end{tikzpicture}\]\\
The bijections \(\tilde\Theta_{i}\) induce two smooth structures on \(\q{\Omega}{\sim}\) such that \(p\) is a surjective submersion, since the \(\Theta_{i}\) are surjective submersions and \(p=\tilde\Theta_{i}^{-1}\circ\Theta_{i}\) by commutativity of the diagram. By uniqueness of the smooth structure on \(\q{\Omega}{\sim}\) these smooth structures must be diffeomorphic, so \(\tilde\Theta_{2}\circ \tilde\Theta_{1}^{-1}\) is a diffeomorphism. Similarly, the bijections \(\tilde\Theta_{i}\) induce two metrics \(g_{1}\) and \(g_{2}\) on \(\q{\Omega}{\sim}\). These metrics are such that the tangents of the curves \(t\mapsto p(\xi,t)\) form a subset with fiberwise nonempty interior of the bundle of future-directed unit timelike vectors in \(g_{i}\). Hence these metrics coincide, and \(\tilde\Theta_{2}\circ \tilde\Theta_{1}^{-1}\) is an isometry. This proves that \((M_1, g_1)\) and \((M_2, g_2)\) are isometric.

\stepheader{Step~\ref{step:Conclude isometry including boundary}}
The manifolds \((\M_1, g_1)\) and \((\M_2, g_2)\) are two conformal completions of the common interior \((M_1, g_1) \cong (M_2, g_2)\). They have strongly causal boundaries and are maximal since they are compact.
By~\cite[Theorem 5.3]{Chrusciel10} and the remarks on page 27 of that paper, these extensions are equivalent, which implies that the manifolds \((\M_1, g_1)\) and \((\M_2, g_2)\) are conformal. Since their interiors are isometric, the conformal factor is \(1\) in the interior and hence everywhere by continuity. This means that \((\M_1, g_1)\) and \((\M_2, g_2)\) are isometric.
\end{proof}

It is possible to extend the above proof to directly yield an isometry including the boundary without using~\cite[Theorem 5.3]{Chrusciel10}. To do this, one works with \(\widehat \Omega\) directly, instead of \(\Omega\), to get a surjection \(\widehat \Omega \to \M\). Since \(\widehat \Omega\) is not necessarily a manifold (\prettyref{fig:Non-manifold Omega}), this is not a submersion in the usual sense, but it retains the property of having sufficiently many local sections for the proof to go through.

The properties of fireworks data which were used in the proof are summarized in the definition of \enquote{generalized fireworks data} given below; the proof goes through without change for any choice of generalized fireworks data. We note this to show that some of the particular choices we made when defining fireworks data do not matter.
\begin{defn}\label{def:Generalized fireworks data}
A \emph{generalized fireworks data map} is a map \(\D\) which to each time-oriented Lorentzian manifold \(\M\) with boundary \(\partial \M\) assigns the pair
\[
\D(\M) = (\TdM \to \partial \M, D),
\]
where \(\TdM \to \partial \M\) is the restriction of the smooth fiber bundle \(T\M \to \M\) to \(\partial \M\) and \(D\) is a subset of \(\TdM \times [0, \infty) \times \TdM\) such that the following holds:\\
There is a family (which is not part of the data) of sets \({(\Sigma_p)}_{p \in \M}\) such that
\begin{itemize}
	\item each set \(\Sigma_p\) is a subset of \(\TdM\),
	\item each set \(\Sigma_p\) consists of future-directed lightlike vectors,
	\item each past-inextendible geodesic ending with some \(\eta \in \Sigma_p\) passes through the point \(p\),
	\item each set \(\Sigma_p\) contains two nonparallel vectors,
	\item \((\xi, t, \eta) \in D\) if and only if
	\begin{itemize}[label={\(\bullet\)}]
		\item \(\xi\) is future-directed and timelike,
		\item \(\eta\) is future-directed and lightlike,
		\item \(\eta \in \Sigma_{\exp(t \hat \xi)}\) where \(\hat \xi\) is the normalization of \(\xi\) to unit length.
	\end{itemize}
\end{itemize}
Moreover, we demand that if \(\M_1\) and \(\M_2\) are isometric, then \(\D(\M_1)\) and \(\D(\M_2)\) are isomorphic in the sense defined below. (If this were not imposed, then one might make \enquote{different} choices of \(\Sigma_p\) for different isometric manifolds.)

The generalized fireworks data involves pairs \((E \to B, D)\) where \(E \to B\) is a smooth fiber bundle and \(D \subset E \times [0, \infty) \times E\) is a set. Our notion of isomorphism between such pairs is the following:
For \(i \in \{1, 2\}\) let \(E_i \to B_i\) be smooth fiber bundles and let \(D_i \subseteq E_i \times [0, \infty) \times E_i\) be arbitrary subsets. We say that the pairs \((E_1 \to B_1, D_1)\) and \((E_2 \to B_2, D_2)\) are \emph{isomorphic} if there is a smooth fiber bundle isomorphism \(F \colon (E_1 \to B_1) \to (E_2 \to B_2)\) such that \((F \times \operatorname{id} \times F)(D_1) = D_2\).
\end{defn}

\section{Broken timelike lens rigidity}
\label{sec:Broken timelike lens rigidity}
We now turn to the question of whether a Lorentzian manifold can be reconstructed from its broken timelike lens data. It turns out that this problem can be reduced to the fireworks rigidity considered in the previous section. 
The broken timelike lens data is defined similarly to the fireworks data, with the difference that we now consider broken geodesics which are always timelike. 
\begin{defn}
Let \(\M\) be a time-oriented Lorentzian manifold with boundary \(\partial \M\). The \emph{broken timelike lens data} of \(\M\) is the set of triples \((\xi, t, \eta) \in \TdM \times \R \times \TdM\) such that
\begin{itemize}
\item \(\xi\) is future-directed and timelike,
\item \(\eta\) is future-directed and timelike,
\item the geodesic starting with \(\xi\) intersects the geodesic ending with \(\eta\), and the broken geodesic thus formed has total length \(t\).
\end{itemize}
The geodesics are allowed to intersect \(\partial \M\) in arbitrarily many points.
We say that \(\M\) and \(\Mprime\) have \emph{isomorphic} broken timelike lens data if \(\TdM \to \partial \M\) and \(\TdMprime \to \partial \Mprime\) are isomorphic as smooth fiber bundles, by an isomorphism which identifies the broken timelike lens data of \(\M\) with the broken timelike lens data of \(\Mprime\)\!, and which identifies the zero section of \(\TdM\) with the zero section of \(\TdMprime\).
\end{defn}

We will now construct the fireworks data of a manifold from its broken timelike lens data, so that we can apply \prettyref{thm:Fireworks data determines isometry class} and get a statement about broken timelike lens rigidity. We do the construction with a lightlike convexity condition on the manifold: We assume that any lightlike geodesic which is tangent to the boundary somewhere must be completely contained in the boundary. In other words, any lightlike geodesic entering the interior of the manifold must intersect the boundary transversely. It is possible to prove broken timelike lens rigidity in greater generality, at the cost of introducing significant further technical complications, by doing this reconstruction with a weaker convexity condition on the manifold. When the convexity condition is weakened, not every lightlike geodesic is the limit of timelike geodesics. However, the proof of \prettyref{thm:Fireworks data determines isometry class} is robust under changes of the data which involve restricting the set of lightlike geodesics represented in the data, as illustrated by the definition of generalized fireworks data in \prettyref{def:Generalized fireworks data}, so we would still get a rigidity statement. 
We will not consider that approach here.
\begin{prop}\label{prop:Broken timelike lens data determines fireworks data}
Let \((\M_1, g_1)\) and \((\M_2, g_2)\) be strongly causal, compact, time-oriented Lorentzian manifolds of dimension \(n \geq 2\) with boundary. Suppose that any lightlike geodesic which is tangent to the boundary \(\partial \M_i\) somewhere is completely contained in the boundary. If \(\M_1\) and \(\M_2\) have isomorphic broken timelike lens data, then they also have isomorphic fireworks data.
\end{prop}
\begin{proof}
Let \(\M\) be a strongly causal, compact, time-oriented Lorentzian manifold with boundary such that any lightlike geodesic which is tangent to the boundary somewhere is completely contained in the boundary. Let \(L\) denote its broken timelike lens data and let \(D\) denote its fireworks data. We will show that \(D\) is the set of all \((\xi, t, \eta)\) such that
\begin{itemize}
\item \(\eta \neq 0\),
\item \((\xi, t, \eta) \notin L\),
\item there is a sequence \(t_i \to t\) and a sequence \(\eta_i \to \eta\) such that \((\xi, t_i, \eta_i) \in L\).
\end{itemize}
We will then have constructed \(D\) from \(L\). The construction is by taking limits, which commutes with diffeomorphisms, so if \(\M_1\) and \(\M_2\) have isomorphic broken timelike lens data then they also have isomorphic fireworks data.

We will now show that \(D\) is the proposed set of limits. Suppose first that \(t_i \to t\) and \(\eta_i \to \eta\) are sequences such that \((\xi, t_i, \eta_i) \in L\) and \((\xi, t, \eta) \notin L\). 
The vector \(\eta\) is future-directed causal since it is the limit of future-directed causal vectors, and non-timelike since otherwise we would have had \((\xi, t, \eta) \in L\). Hence \(\eta\) is lightlike. 
That \((\xi, t_i, \eta_i) \in L\) means that there are \(\tau_i, s_i \geq 0\) such that
\[\exp(\tau_i \xi) = \exp(-s_i \eta_i),\]
\[\tau_i \sqrt{-g(\xi, \xi)} + s_i \sqrt{-g(\eta_i, \eta_i)} = t_i.\]
We will prove that \(s_i\) has a bounded subsequence. If not, then \(s_i \to \infty\). Then \(\exp(-s \eta_i) \to \exp(-s \eta)\) for all \(s \geq 0\) by continuity of the exponential map, so \(s \mapsto \exp(-s \eta)\) is a geodesic defined on \([0, \infty)\). This geodesic is lightlike since \(\eta\) is lightlike. This is a contradiction since \(\M\) is strongly causal. This means that \(s_i\) has a bounded subsequence, and hence also a convergent subsequence. For notational convenience, assume without loss of generality that \(s_i\) itself is convergent, and denote its limit by \(s\). This means that \(s_i \sqrt{-g(\eta_i, \eta_i)} \to 0\) so that \(\tau_i \sqrt{-g(\xi, \xi)} \to t\). Let \(\tau = t/\sqrt{-g(\xi, \xi)}\). We now know that
\[\exp(\tau \xi) = \exp(-s \eta),\]
\[\tau \sqrt{-g(\xi, \xi)} + s \sqrt{-g(\eta, \eta)} = t.\]
Hence there is a broken geodesic of length \(t\) starting with \(\xi\) and ending with \(\eta\). This shows that \((\xi, t, \eta) \in D\). 

Conversely, suppose that \((\xi, t, \eta) \in D\).
Let \(p = \exp(t \hat \xi)\), where \(\hat \xi\) is the normalization of \(\xi\) to unit length. 
If \(p \in \partial \M\) then \((\xi, t, \zeta) \in L\) for all timelike future-directed \(\zeta \in T_p\M\) so in particular there is a sequence \(\zeta_i \to \eta\) such that \((\xi, t, \zeta_i) \in L\). Suppose now that \(p \notin \partial \M\).
Let \(\zeta \in T_p\M\) be such that the geodesic starting with \(\zeta\) ends with \(\eta\). By hypothesis, this geodesic must be transverse to \(\partial \M\). By \prettyref{lem:Smooth time to boundary} all timelike vectors in \(T_p\M\) sufficiently close to \(\zeta\) also give geodesics which are transverse to \(\partial \M\), and the function \(T_\partial\) measuring the parameter time to the boundary is smooth near \(\zeta\). Choose a sequence \(\zeta_i \to \zeta\) and let \(\eta_i = \Phi_g(T_\partial(\zeta_i), \zeta_i)\), where \(\Phi_g\) denotes the geodesic flow. Let \(t_i = t + \sqrt{-g(\zeta_i, \zeta_i)} T_\partial(\zeta_i)\). Then \((\xi, t_i, \eta_i) \in L\) and \((\xi, t_i, \eta_i) \to (\xi, t, \eta)\). Since \(\eta\) is lightlike, it cannot be the case that \((\xi, t, \eta) \in L\). This completes the proof.
\end{proof}

\begin{figure}
 \centering{}
 \begin{overpic}[width=0.27\textwidth]{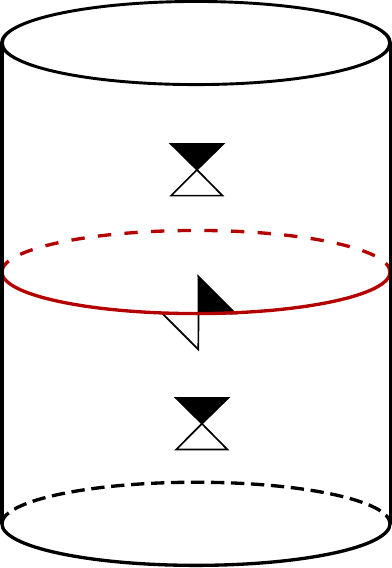}
 \end{overpic}
 \caption{A metric where the construction in \prettyref{prop:Broken timelike lens data determines fireworks data} of fireworks data from broken timelike lens data does not work, since it captures far too many lightlike vectors at the upper boundary. This example violates the condition of strong causality.}
\label{fig:Counterexample to limit construction}
\end{figure}

The above construction does not work if the condition that the manifold is strongly causal is completely removed. As a counterexample, consider a metric on \(\left[0,1\right]\times S^{1}\) with a single closed lightlike curve and spacelike boundary, as shown in \prettyref{fig:Counterexample to limit construction}. Then for every future-directed timelike vector \(\xi\) at the lower boundary, there is a \(t\) such that \(\left(\xi,t,\eta\right)\) is a limit of the broken timelike lens data for all lightlike \(\eta\) at the upper boundary. However, the fireworks data contains at most two such triples for each pair \(\left(\xi,t\right)\).

We can now combine \prettyref{prop:Broken timelike lens data determines fireworks data} with the result from the previous section to obtain a theorem about broken timelike lens rigidity.

\begin{thm}\label{thm:Broken timelike lens rigidity}
Let \(\left(\M,g\right)\) be a strongly causal, compact, time-oriented Lorentzian manifold of dimension \(n \geq 3\) with boundary, where every point belongs to a past-inextendible timelike geodesic intersecting the boundary transversely. Suppose moreover that any lightlike geodesic which is tangent to the boundary somewhere is completely contained in the boundary. The isometry type of \(\left(\M,g\right)\) is then uniquely determined by the broken timelike lens data of \((\M, g)\).
\end{thm}
\begin{proof}
Combine \prettyref{prop:Broken timelike lens data determines fireworks data} and \prettyref{thm:Fireworks data determines isometry class}.
\end{proof}

\section{Sky shadow rigidity}
\label{sec:Sky shadow rigidity}
We will now discuss another kind of data at the boundary of a Lorentzian manifold with boundary. Suppose that lightlike geodesics are sent out from a point. These will typically intersect the boundary in a lower-dimensional set, giving the \enquote{shadow} which would be formed if light rays were prevented from traveling through the point. The word \enquote{sky} is used in~\cite{Low88} to mean the set of lightlike geodesics through a point, so these shadows are in a sense shadows of skies on the boundary. The question is then, if we observe the shadows, can we reconstruct the manifold and its metric? In this section we will make the question precise, and show that indeed we can reconstruct the manifold and its metric up to a conformal factor, provided we impose the strong causality condition, a weak convexity condition, and demand that no lightlike geodesics have conjugate points.

We will begin by making the concept of a \enquote{sky shadow} precise.
\begin{defn}
Let \(\M\) be a Lorentzian manifold with boundary. The \emph{pre-sky} of a point \(p \in \M\) is the set of vectors \(\eta \in T\M\) such that \(\eta\) is tangent to a lightlike geodesic through \(p\).
\end{defn}
We call this a pre-sky since it is the preimage of a sky under the projection from the space of lightlike vectors to the space of skies. Intersecting the past half of the pre-sky of a point with the tangent bundle over the boundary gives the \enquote{past sky shadow} of the point. In other words it may be defined as follows.
\begin{defn}
The \emph{past sky shadow} of a point \(p \in \M\) is the set of \(\eta \in \TdM\) such that either \(\eta\) or \(-\eta\) is the initial velocity of a future-directed lightlike geodesic through \(p\). As previously, \(\TdM\) is the restriction of \(T\M\) to the boundary of \(\M\).
\end{defn}
Some past sky shadows are illustrated in \prettyref{fig:Past sky shadow} and \prettyref{fig:Past sky shadow 3D}.
\begin{defn}
The \emph{past sky shadow data} of a time-oriented Lorentzian manifold \(\M\) with boundary is the set of past sky shadows of points in the interior of \(\M\). We say that \(\M\) and \(\Mprime\) have \emph{isomorphic} past sky shadow data if \(\TdM \to \partial \M\) and \(\TdMprime \to \partial \Mprime\) are conformal as smooth vector bundles with Lorentzian metrics, by a conformal vector bundle isomorphism which
\begin{itemize}
\item identifies the tangent bundles of the boundaries, that is identifies \(T\partial \M\) with \(T\partial \Mprime\),
\item identifies the past sky shadow data of \(\M\) with the past sky shadow data of \(\Mprime\)\!.
\end{itemize}
\end{defn}
Note that this is a stronger notion of isomorphism than the one we used for fireworks data.

We begin by showing by an example that the past sky shadow data is in general not sufficient to determine a time-oriented Lorentzian manifold uniquely unless additional conditions are imposed.
\begin{example}
\label{ex:Periodic spacetime}
Let \(n \geq 3\) and let \(\left(S^{n-1},h\right)\) be the \((n-1)\)-sphere with the standard round unit sphere Riemannian metric \(h\), so that all unit speed geodesics are periodic with period \(2\pi\). For each positive real number \(T\) let
\[
\left(\M_{T},g_{T}\right)=\left(\left[0,T\right]\times S^{n-1},-dt^{2}+h\right).
\]
Then the lightlike geodesics through any point \(\left(t,p\right)\in \M_{T}\) are the same as the lightlike geodesics through \(\left(t+2\pi k,p\right)\) for every integer \(k\) (provided, of course, that \(t+2\pi k\in\left[0,T\right]\)). This means that \(\M_{T}\) has the same sky shadow data as \(\M_{T + 2\pi k}\) for all positive integers \(k\), even though they are non-conformal. Note that these manifolds \(\M_{T}\) are well-behaved in the sense that they are globally hyperbolic and have spacelike boundary. Note also that when \(T\) is sufficiently small, no lightlike geodesic in \(\M_T\) contains a pair of conjugate points.
\end{example}
We therefore need more information about a manifold to be able to reconstruct it from the sky shadow data. A condition that is technically convenient to work with is the condition that no lightlike geodesic contains a pair of conjugate points. In the family \({(\M_{T + 2 \pi k})}_{k \in \Z}\) in the above example, there is precisely one manifold in which no lightlike geodesics contain conjugate points, so if we know a priori that the manifold which we want to reconstruct has this property, then we might hope to be able to reconstruct it.

We are now ready to state and prove our second main theorem, which tells us that the past sky shadow data defined above uniquely determines the manifold structure and metric up to a conformal factor, provided the strong causality condition holds and no lightlike geodesics contain conjugate points, and provided we impose a weak convexity condition.
\begin{thm}\label{thm:Past sky shadow data determines conformal class}
Let \((\M_1, g_1)\) and \((\M_2, g_2)\) be strongly causal, compact, time-oriented Lorentzian manifolds of dimension \(n \geq 3\) with boundary. Suppose that both manifolds are such that
\begin{itemize}
\item no lightlike geodesic contains a pair of conjugate points, and
\item every interior point lies on a past-inextendible lightlike geodesic which intersects the boundary transversely.
\end{itemize}
If \(\M_1\) and \(\M_2\) have isomorphic past sky shadow data, then they are conformal.\\
In other words, the conformal class of any strongly causal, compact, time-oriented Lorentzian manifold of dimension \(n \geq 3\) with boundary with the above conditions on lightlike geodesics is uniquely determined by the past sky shadow data.
\end{thm}

As was the case with \prettyref{thm:Fireworks data determines isometry class}, the proof can be extended somewhat to cover some settings in which the manifold is not necessarily compact and the data is not necessarily given at the boundary. However, we consider only the case of compact manifolds with data given at the boundary.
\begin{proof}
The idea of the proof is to use the observation from \prettyref{sec:Reconstruction of smooth structures}, and construct a suitable manifold \(\Omega\) from the past sky shadow data of \(M_1\) and \(M_2\). This time, as opposed to in the proof of \prettyref{thm:Fireworks data determines isometry class}, we will not explicitly define an equivalence relation on \(\Omega\), and the submersions \(\Omega \to M_i\) will factor through vector bundles over \(M_i\). 

Let \(\TdM \to \partial \M\) denote a smooth vector bundle which is isomorphic to both \({\TdM_1 \to \partial \M_1}\) and \({\TdM_2 \to \partial \M_2}\). We introduce this bundle to emphasize that the construction of \(\Omega\) is independent of \(i\). Similarly, let \([g]\) denote the conformal class of the indefinite inner products on \(\TdM\) given by the isomorphisms with \(\TdM_1\) and \(\TdM_2\), and view \(T\partial \M\) as a subset of \(\TdM\). This is well-defined by definition of isomorphism of past sky shadow data. Let \(\NdM \to \partial \M\) denote the subbundle of \(\TdM\) consisting of lightlike vectors (which can be identified using the conformal class \([g]\)). 

We will begin by constructing a fiber bundle \(J\) which will be needed for the argument. This bundle \(J\) will consist of pairs \((\eta, b)\) where \(\eta \in \NdM\) is a lightlike vector based at some point in \(\partial \M\) and \(b\) is a vector space endomorphism of \(\qextraspace{\eta^\perp}{\eta}\). Here \(\eta^\perp\) denotes the space of vectors orthogonal to \(\eta\), as defined in \prettyref{app:Null geometry}. We have \(\R\eta \subseteq \eta^\perp\) since \(\eta\) is null, so the quotient makes sense. The explicit construction of \(J\), which will endow it with a fiber bundle structure, is as follows. View \({\NdM \oplus \TdM}\) as a vector bundle over \(\NdM\) and let \(\mathfrak P \subseteq \NdM \oplus \TdM\) be the subbundle defined by
\[\mathfrak P = \{(\eta, X) \in \NdM \oplus \TdM \mid [g](\eta, X) = 0\}.\]
Define an equivalence relation \(\asymp\) on \(\mathfrak P\) by letting \((\eta, X) \asymp (\eta, X + \lambda \eta)\) for all \(\lambda \neq 0\). Then \(\q{\mathfrak P}{\asymp}\) is a vector bundle over \(\NdM\). Let \(J\) be the endomorphism bundle of this vector bundle \(\q{\mathfrak P}{\asymp}\). Then \(J\) consists of pairs \((\eta, b)\) where \(\eta \in \NdM\) and \(b\) is a vector space endomorphism of \(\qextraspace{\eta^\perp}{\eta}\), as promised.

The past sky shadows are subsets of \(\TdM\). In the nicest case, a past sky shadow (after adding a zero section) is a line bundle over an embedded spacelike hypersurface \(\Sigma \subseteq \partial \M\), consisting of lightlike vectors which are orthogonal to \(\Sigma\). Even if this does not quite hold, the past sky shadows may have subsets of this form. We will collect these subsets in a set \(\mathcal S\). More precisely, let \(\mathcal S\) denote the set of line bundles \(N \to \Sigma\) where \(\Sigma \subseteq \partial \M\) is an embedded spacelike hypersurface and the total space \(N\) (after removing the zero section) is a subset of a past sky shadow such that all nonzero elements of \(N\) are transverse to \(\partial \M\) and orthogonal to \(\Sigma\). This set \(\mathcal S\) can be determined since we have an inclusion \(T\partial \M \hookrightarrow \TdM\) and a conformal metric for \(\TdM\). 
We will use the elements of \(\mathcal S\) to construct an interesting subset \(\Omega \subseteq J\). Consider an element \((N \to \Sigma) \in \mathcal S\). The image under the exponential map of a sufficiently small neighborhood of the zero section of \(N\) is an immersed null hypersurface \(\mathcal H\). For each \(\eta \in N\) we may compute the null Weingarten map (defined in \prettyref{app:Null geometry}) of \(\mathcal H\) with respect to \(\eta\). Call this null Weingarten map \(b(\eta)\). We define the set \(\Omega\) using an auxiliary set \(\widehat\Omega\): For each vector \(\eta \in N\), let \((\eta, b(\eta)) \in \widehat\Omega\). By doing this for all elements \((N \to \Sigma) \in \mathcal S\), we have constructed the set \(\widehat\Omega \subseteq J\).
With \(\pi_1 \colon (\eta, b) \mapsto \eta\), define
\[\Omega = \left\{(\eta, b) \in \widehat\Omega \mid \eta \text{ is an interior point of } \pi_1(\widehat\Omega) \subseteq \NdM \right\}.\]
Note that rescaling a lightlike vector \(\eta\) by some \(\lambda \neq 0\) rescales the null Weingarten map \(b(\eta)\) by the same factor \(\lambda\), so if \((\eta, b) \in \Omega\) then \((\lambda\eta, \lambda b) \in \Omega\) for all \(\lambda \neq 0\).

For \(i \in \{1, 2\}\), let \(M_i = \interior \M_i\) be the interior of \(\M_i\). Let \(NM_i\) denote the lightlike tangent bundle of \(M_i\), and let \(\oNM_i = NM_i \cap \oTM_i\), where \(\oTM_i\) is defined as in \prettyref{lem:Smooth time to boundary}. Note that \(\oNM_i\) is an open subset of \(NM_i\) since \(\oTM_i \subseteq TM_i\) is open.
We will now show that \(\Omega\) is an embedded submanifold of \(J\), diffeomorphic to both \(\oNM_1\) and \(\oNM_2\). To that end, fix \(i \in \{1, 2\}\).

\newcounter{stepnumber}
\newcommand{\step}{\protect\stepcounter{stepnumber}Step~\thestepnumber}

\stepheader{\step: Definition of a smooth map \(\Psi \colon \oNM_i \to J\) with \(\im(\Psi) = \Omega\).}
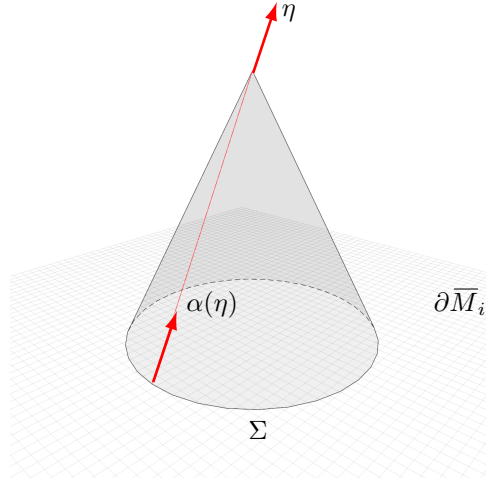
\begin{figure}
 \centering{}
 \begin{overpic}[width=0.5\textwidth]{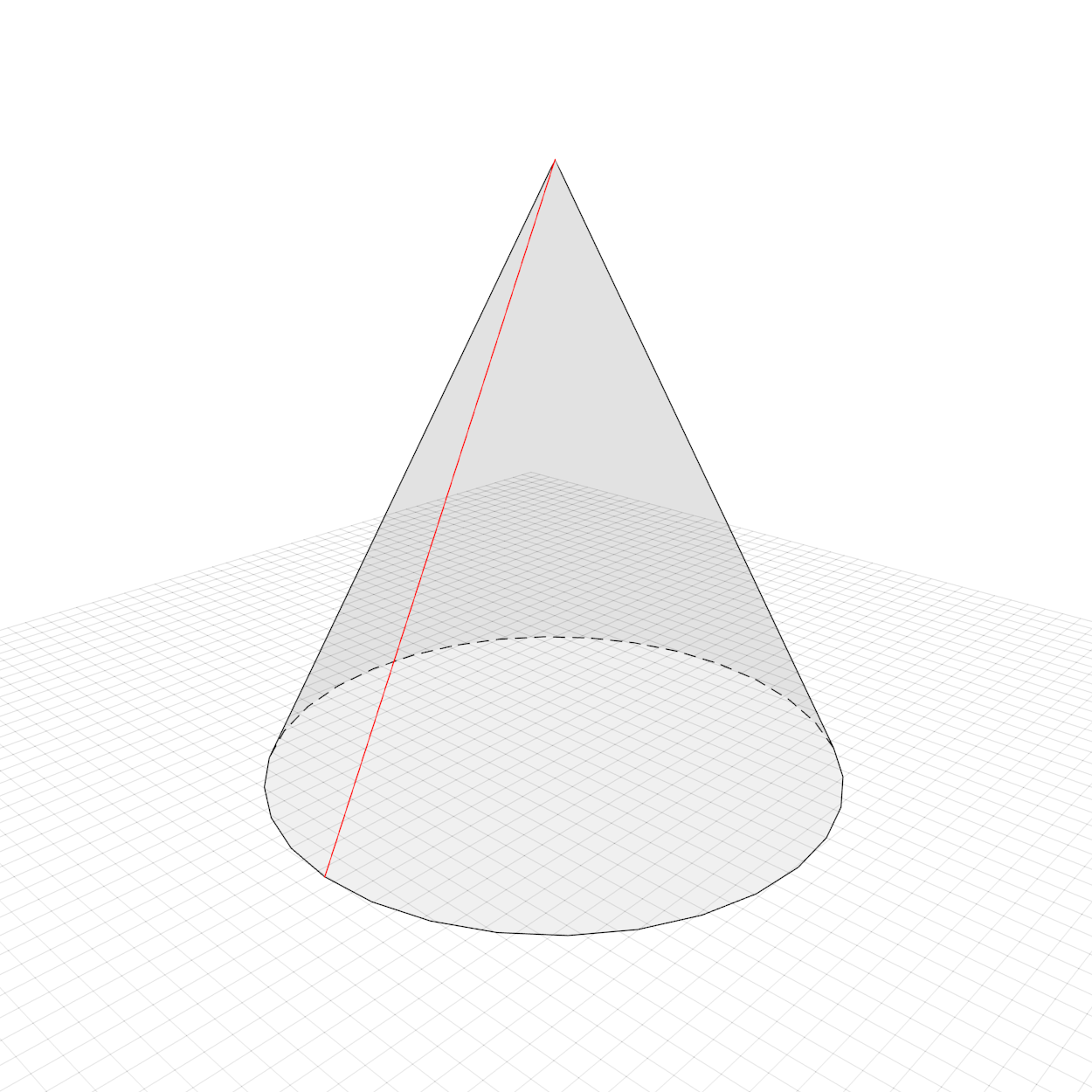}
  \put (57,97) {\(\eta\)}
  \put (37,35) {\(\alpha(\eta)\)}
  \put (50,8) {\(\Sigma\)}
  \put (89,35) {\(\partial \M_i\)}
  \linethickness{1pt}
  \put(30,20){\color{red}\vector(1.5,4.5){5}}
  \put(51,85){\color{red}\vector(1.5,4.5){5}}
 \end{overpic}
 \caption{The map \(\alpha\) flows each lightlike vector to the boundary, along the geodesic it generates. The past sky shadow of the base point of \(\eta\) is a line bundle over the spacelike submanifold \(\Sigma\).}
\label{fig:The map alpha}
\end{figure}
Consider the past lightlike half-geodesics through the base point of a vector \(\eta \in \oNM_i\), as shown in \prettyref{fig:The map alpha}. These half-geodesics form a null hypersurface which intersects the boundary of the manifold in a set \(\Sigma\). By parallel transport of \(\eta\) to the boundary along the geodesic it generates we obtain a vector \(\alpha(\eta)\), and we can compute the null Weingarten map \(b(\alpha(\eta))\) of the null hypersurface with respect to \(\alpha(\eta)\). This defines a map \(\Psi(\eta) = (\alpha(\eta), b(\alpha(\eta)))\). We will now prove that this map is smooth and that \(\im(\Psi) = \Omega\).

Let \(\Phi_g\) denote geodesic flow, let \(T_\partial\) be defined as in \prettyref{lem:Smooth time to boundary} and let \(\pi\) denote the projection \(\pi \colon T\M_i \to \M_i\). 
Define \(\mathcal T_\partial\) by
\[
\mathcal T_\partial(\eta)
=
\begin{cases}
-T_\partial(-\eta) \quad&\text{if \(\eta\) is future-directed,}\\
T_\partial(\eta) \quad&\text{if \(\eta\) is past-directed.}
\end{cases}\]
The map \(\alpha \colon \oNM_i \to \TdM_i\) can then be described by \(\alpha(\eta) = \Phi_g(\mathcal T_\partial(\eta), \eta)\).
This map flows lightlike vectors to the boundary, as shown in \prettyref{fig:The map alpha}.
The map is smooth since \(\mathcal T_\partial\) is smooth by \prettyref{lem:Smooth time to boundary}. In particular, for each \(p \in M_i\) the restriction \(\alpha_p\) of \(\alpha\) to \(\oNpM_i\) is smooth, and we interpret \(\alpha\) as a smooth family of smooth maps \(\alpha_p\) indexed by \(p \in M_i\). In particular, the derivatives of \(\alpha_{\pi(\eta)}\) at \(\eta\) depend smoothly on \(\eta\). Fix \(p \in M_i\) and let \(\mathfrak N = \oNpM_i \cap \mathcal T_\partial^{-1}(1)\). 
The restriction of \(\mathcal T_\partial\) to \(\oNpM_i\) is a submersion since \(\lambda \mathcal T_\partial(\lambda X) = \mathcal T_\partial(X)\) for all \(\lambda > 0\) and all \(X \in \oNpM_i\). This means that \(\mathfrak N \subset \oNpM_i\) is an embedded submanifold. The restriction of \(\alpha_p\) to the future-directed component of \(\mathfrak N\) is geodesic flow for time \(-1\), which is an embedding. The restriction of \(\pi \circ \alpha_p\) to \(\mathfrak N\) is simply the exponential map composed with negation, which is an immersion since we have assumed that no lightlike geodesics have conjugate points. Now \(\alpha_p(\oNpM_i)\) is a line bundle (except for missing a zero section) over the immersed hypersurface \(\pi(\alpha_p(\oNpM_i)) \subset \partial \M_i\), consisting of vectors which are lightlike and normal to the hypersurface. For each \(\eta \in \oNM_i\), we define \(\Psi(\eta) = (\alpha(\eta), b(\alpha(\eta))) \in \Omega\). 
Note that the image of \(\alpha_{\pi(\eta)}\) is a spacelike section of the null hypersurface for which \(b(\alpha(\eta))\) is the null Weingarten map. This means that \(b(\alpha(\eta))\) is a smooth function of the derivatives of \(\alpha_{\pi(\eta)}\) at \(\eta\). These derivatives depend smoothly on \(\eta\), so we conclude that \(\Psi(\eta)\) depends smoothly on \(\eta\).

To show that the image of \(\Psi\) is all of \(\Omega\), we first show that each \(\alpha_p\) is an embedding, so that every embedded \((n-2)\)-submanifold in the image of \(\alpha_p\) is the diffeomorphic image of a submanifold of \(\oNpM_i\). Since \(\alpha_p(\lambda\eta) = \lambda \alpha_p(\eta)\) for \(\lambda > 0\) and since \(\mathfrak N\) is transverse to all curves \(\lambda \mapsto \lambda \eta\), we get diffeomorphic splittings \(\oNpM \cong \mathfrak N \times \R^+\) and \(\alpha_p(\oNpM) \cong \alpha_p(\mathfrak N) \times \R^+\), in terms of which we can express \(\alpha_p((\eta, \lambda)) = (\alpha_p(\eta), \lambda)\). Since \(\left.\alpha_p\right|_{\mathfrak N}\) is an embedding, \(\alpha_p\) is also an embedding.
To see that \(\im(\Psi) = \Omega\), choose \((\eta, b) \in \Omega\). Since \(\Psi(\zeta) = -\Psi(-\zeta)\) for all \(\zeta\), we may suppose without loss of generality that \(\eta\) is future-directed. By definition of \(\Omega\) the pair \((\eta, b)\) comes from the past sky shadow of some point \(p \in M_i\). Let \(\zeta \in N_pM_i\) be such that there is a geodesic starting with \(\eta\) and ending with \(\zeta\). This geodesic intersects \(\partial \M_i\) transversely since \(\eta\) is an interior point of \(\pi(\widehat\Omega)\), so \(\zeta \in \oNpM_i\). Then by definition of \(\Psi\) we have \(\Psi(\zeta) = (\eta, b)\).

\stepheader{Interlude: A characterization of an inverse of \(\Psi\).}
In the next step we will show that \(\Psi\) is injective. Its inverse can be characterized using the optical equation for the null Weingarten map. We begin by describing the idea. Consider a congruence of past-inextendible future-directed lightlike geodesics through a point \(p \in M_i\). Let \(\gamma \colon [0, T] \to \M_i\) be one of the geodesics, with an affine parameterization where \(\gamma(0) \in \partial \M_i\) and \(\gamma(T) = p\). For \(t \in [0, T]\) let \(b(t)\) be the null Weingarten map of the congruence along \(\gamma\). Then \(\Psi(\dot\gamma(T)) = (\dot\gamma(0), b(0))\). Our task is to reconstruct \(\dot\gamma(T)\) from \((\dot\gamma(0), b(0))\). Of course, reconstructing \(\gamma\) from an initial velocity \(\dot\gamma(0)\) is simply a matter of existence and uniqueness of geodesics. The nontrivial part is to reconstruct the time \(T\). To do this, we will use the optical equation from \prettyref{app:Null geometry}, which says that
\[b'(t) = -{b(t)}^2 - R_{\dot\gamma(t)}\]
with \('\) denoting covariant derivative in direction \(\dot\gamma(t)\).
This, together with initial data \(b(0)\), forms an initial value problem, which has a unique solution. The idea is now that something very significant happens to \(b(t)\) as \(t \to T\); a computation in normal coordinates centered at the apex \(p = \gamma(T)\) shows that
\[\lim_{t \to T^-} \tr b(t) = -\infty.\]
For future use we note that the computation also shows that the trace-free part of \(b(t)\) tends to zero:
\[\lim_{t \to T^-} \left(b(t) - \frac{\tr b(t)}{n - 2} \id\right) = 0.\]
This means that the existence interval for the solution to the initial value problem cannot include \(T\). On the other hand, the null Weingarten map of the geodesic congruence is well defined and satisfies the differential equation for all \(t \in [0, T)\), so the existence interval must be precisely \([0, T)\). This is how we can reconstruct \(T\) from \((\dot\gamma(0), b(0))\). We will now rephrase this in more precise language.

Given \((\eta, b_0) \in \Omega\), we wish to define \(\Psi^{-1}(\eta, b_0) \in \oNM_i\). The function \(\Psi^{-1}\) will later be shown to be an inverse to \(\Psi\). We begin by defining \(\Psi^{-1}(\eta, b_0)\) for future-directed \(\eta\), and later extend the definition by letting \(\Psi^{-1}(-\eta, -b_0) = -\Psi^{-1}(\eta, b_0)\). Let \(\gamma_\eta\) be the inextendible geodesic with \(\dot\gamma_\eta(0) = \eta\). 
The following is an initial value problem for a tensor along \(\gamma_\eta\), with \('\) denoting covariant derivative in direction \(\dot\gamma_\eta\):
\begin{equation}
\label{eq:Optical initial value problem}
\tag{$\star$}
\begin{cases}
b'(t) = -{b(t)}^2 - R_{\dot\gamma_\eta(t)},\\
b(0) = b_0.
\end{cases}
\end{equation}
Let \(T\) be the existence time of the maximal solution to this initial value problem and define \(\Psi^{-1}(\eta, b_0) = \dot \gamma_\eta(T)\). Note that \(\Psi^{-1}\) commutes with multiplication by \(\lambda > 0\) for geometrical reasons; multiplication by positive real numbers corresponds to a change of parameterization of \(\gamma\). We extend \(\Psi^{-1}\) to the case when \(\eta\) is past-directed by letting \(\Psi^{-1}(\eta, b_0) = -\Psi^{-1}(-\eta, -b_0)\). Then \(\Psi^{-1}\) commutes with multiplication by any \(\lambda \neq 0\).

For general initial data the existence interval could be cut off by the boundary of the manifold. This would happen for instance if we let \(\M_i\) be a compact part of Minkowski space, and use initial data \(b_0 = 0\). However, for initial data from \(\Omega\) we know a priori that the existence interval will end due to the trace of \(b\) diverging.

\stepheader{\step: \(\Psi\) is injective, and its inverse is continuous.}
Since both \(\Psi\) and \(\Psi^{-1}\) commute with multiplication by nonzero real numbers, we may assume without loss of generality that the vectors \(\eta\) considered below are future-directed.

We will now show that the function \(\Psi^{-1}\) defined above is actually an inverse to \(\Psi\), and that it is continuous.
That \((\eta, b_0)\) is in the image of \(\Psi\) means that it comes from a pre-sky of some point \(p \in M_i\). The projection of this pre-sky by \(\pi \colon TM_i \to M_i\), or equivalently the image under the exponential map of the light cone at \(p\), is an immersed null hypersurface, since we have assumed that no lightlike geodesics have conjugate points. Letting \(\gamma\) be the geodesic with \(\dot\gamma(0) = \eta\) and \(\gamma(T) = p\), and letting \(b\) be the null Weingarten map of the immersed null hypersurface along \(\gamma\), we get a solution to the initial value problem used to define \(\Psi^{-1}\). This solution is bounded on compact subsets of \([0, T)\) and blows up at \(T\), which means that \(\Psi^{-1}(\eta, b_0) = \dot\gamma(T)\), proving that \(\Psi^{-1}\) is an inverse to \(\Psi\). Note that this was easy to show only because the projected pre-sky is an immersed hypersurface, for which we needed that there are no conjugate points. 

The function \(\Psi^{-1}\) can be written as
\[\Psi^{-1}(\eta, b_0) = \Phi_g(T(\eta, b_0), \eta)\]
where \(\Phi_g\) is the geodesic flow and \(T(\eta, b_0)\) is the existence time of solutions to the initial value problem~\eqref{eq:Optical initial value problem}. To see that \(\Psi^{-1}\) is continuous, it is then sufficient to show that \(T(\eta, b_0)\) depends continuously on \(\eta\) and \(b_0\). We will remove all geometry from the problem and reduce it to a question about ordinary differential equations. Let \(\gamma_\eta\) be the geodesic with \(\dot\gamma_\eta(0) = \eta\). The null Weingarten map along \(\gamma\) is an endomorphism of \(\qextraspace{T^\perp\gamma_\eta}{T\gamma_\eta}\), as discussed in \prettyref{app:Null geometry}. Choose a parallel frame \(X_1(t, \eta), \ldots, X_{n-2}(t, \eta)\) for \(\qextraspace{T^\perp\gamma_\eta}{T\gamma_\eta}\) along \(\gamma_\eta\) and let \(R(t, \eta)\) be the matrix representation of \(R_{\dot\gamma_\eta(t)}\) in this frame. Note that \(X_i(t, \eta)\) and \(R(t, \eta)\) can be chosen to depend smoothly on \(t\) and \(\eta\). Using this frame, the initial value problem~\eqref{eq:Optical initial value problem} can be written as an initial value problem for matrices:
\[\begin{cases}
\dot b(t, \eta, b_0) = -b^2(t, \eta, b_0) - R(t, \eta),\\
b(0, \eta, b_0) = b_0.
\end{cases}\]
Here we view \(\eta\) and \(b_0\) as parameters for the initial value problem, and \(\dot b\) is the derivative of \(b\) with respect to \(t\). Let \(\theta\) and \(\sigma\) be the trace and trace-free part of \(b\), respectively:
\[\theta = \tr(b),\]
\[\sigma = b - \frac{\theta}{n-2} \id.\]
The differential equation for \(b\) splits into an equation for \(\theta\) and an equation for \(\sigma\) as follows:
\[\begin{cases}
\dot \theta(t, \eta, b_0) = -\frac{\theta^2}{n - 2} - \tr(\sigma^2) - r(t, \eta),\\
\dot \sigma(t, \eta, b_0) = F(\theta, \sigma, \eta, t).
\end{cases}\]
We have suppressed the arguments \((t, \eta, b_0)\) in the right hand side of the first equation for clarity. The exact form of the right hand side of the second equation is not important for our argument, so we have encapsulated it into a function \(F\). The only property of \(F\) we will use is that it is smooth. We have used \(r(t, \eta)\) to denote the trace of \(R(t, \eta)\), which is smooth. Recall that the trace of \(b\) diverges to \(-\infty\) at the endpoint of the existence interval (for initial data in \(\Omega\)), so we want to know when \(\theta \to -\infty\). We change variables by letting
\[a = \arctan(\theta).\]
The resulting system of differential equations is
\[\begin{cases}
\dot a(t, \eta, b_0) = -\frac{\sin^2(a)}{n - 2} - \cos^2(a)\left(\tr(\sigma^2) + r(t, \eta)\right),\\
\dot \sigma(t, \eta, b_0) = F(\tan(a), \sigma, \eta, t).
\end{cases}\]
What we are interested in is the time \(T(\eta, b_0)\) after which \(a \to -\pi/2\). As long as \(a \neq \pm\frac{\pi}{2}\) and \(\sigma\) is bounded, this is a smooth system of ordinary differential equations. This means that \(a\) and \(\sigma\) are smooth functions on the set
\[\Gamma = \left\{(t, \eta, b_0) \mid 0 < t < T(\eta, b_0)\right\},\]
and moreover that \(\Gamma\) is open.
We have
\begin{equation}
\label{eq:Finiteness of a}
a(t, \eta, b_0) > -\frac{\pi}{2} \quad \text{ if }\ 0 < t < T(\eta, b_0).
\end{equation}
We can also derive an upper bound (local in \(\eta\)) for the existence time when the initial value for \(a\) is sufficiently close to \(-\frac{\pi}{2}\). Note that \(r(t, \eta) = \Ric(\dot\gamma_\eta(t), \dot\gamma_\eta(t))\) where \(\gamma_\eta\) is the geodesic with initial velocity \(\eta\). This means that for \(\xi\) close to \(\eta\) the functions \(r(\cdot, \xi)\) are uniformly bounded, say by \(C > 0\). Moreover, \(\sigma\) is symmetric so \(\tr(\sigma^2) \geq 0\). Hence the differential equation for \(a\) implies that
\[\dot a(t, \xi, b_0) < -\frac{\sin^2(a)}{n - 2} + C\cos^2(a)\]
for \(\xi\) in a neighborhood of \(\eta\). This means for each \(\epsilon > 0\) there is a \(\delta > 0\) such that
\begin{equation}\label{eq:Existence bound}
a(\tau, \xi, b_0) < -\frac{\pi}{2} + \delta \implies \dot a(t, \xi, b_0) < -1 + \epsilon \text{ for all } t > \tau
\end{equation}
for all \(b_0\) and for \(\xi\) in a neighborhood of \(\eta\).
For \(\epsilon < 1\) this means that \(a\) must reach the value \(-\frac{\pi}{2}\) within time \((a(\tau, \xi, b_0) + \frac{\pi}{2})/(1 - \epsilon)\). Summarizing, if \(a(\tau, \xi, b_0) < -\frac{\pi}{2} + \delta\) for some \(\tau\) then \(T(\xi, b_0) < \tau + (a(\tau, \xi, b_0) + \frac{\pi}{2})/(1 - \epsilon)\).

To see that \(T\) is continuous on \(\Omega\), consider a convergent sequence \(x_i \to x\) in \(\Omega\). (For clarity of notation, we now abbreviate pairs of parameters \((\eta, b_0)\) by single letters.) Then the sequence \({(T(x_i))}_{i=1}^\infty\) must have a convergent subsequence in the sequentially compact space \([0, \infty]\). We will show that each such subsequence has the limit \(T(x)\), thereby proving that \(T\) is continuous. Pass to a subsequence and let
\[T_\infty = \lim_{i \to \infty} T(x_i).\]
We will show that \(T_\infty = T(x)\). If it were the case that \(T_\infty < T(x)\) then \((T_\infty, x)\) would belong to the open set \(\Gamma\). The function \(a\) is continuous on this set so then it would hold that
\[a(T_\infty, x) = \lim_{i \to \infty} a(T(x_i), x_i) = -\frac{\pi}{2}.\]
However by~\eqref{eq:Finiteness of a} this means that \(T_\infty \geq T(x)\), contradicting the assumption that \(T_\infty < T(x)\). We will now show that the reverse inequality also holds, so suppose for contradiction that \(T_\infty > T(x)\). 
Let \(\epsilon > 0\) and let \(\delta > 0\) be such that~\eqref{eq:Existence bound} holds in a neighborhood of \(x\). Let \(\beta \in (0, \delta)\) be arbitrary. Since
\[\lim_{t \to {T(x)}^-} a(t, x) = -\frac{\pi}{2}\]
we may choose \(\mathfrak t < T(x)\) depending on \(\beta\) such that \(a(\mathfrak t, x) < -\frac{\pi}{2} + \frac{\beta}{2}\).
By continuity of \(a\) on \(\Gamma\) it then holds that
\[\lim_{i \to \infty} a(\mathfrak t, x_i) < -\frac{\pi}{2} + \frac{\beta}{2}\]
so for all sufficiently large values of \(i\) it holds that \(a(\mathfrak t, x_i) < -\frac{\pi}{2} + \beta\). Hence by~\eqref{eq:Existence bound} and by choice of \(\epsilon\), \(\delta\), and \(\beta\)
\[T(x_i) < \mathfrak t + \frac{\beta}{1 - \epsilon}.\]
Taking the limit as \(i \to \infty\) we obtain
\[T_\infty \leq \mathfrak t + \frac{\beta}{1 - \epsilon}.\]
Since \(\beta \in (0, \delta)\) was arbitrary and \(\mathfrak t < T(x)\) for every \(\beta\), this means that
\[T_\infty \leq T(x).\]
This completes the proof that \(T_\infty = T(x)\), proving that \(T\) is indeed continuous.

We have now shown that the existence time of solutions to the initial value problem~\eqref{eq:Optical initial value problem} depends continuously on the initial values and parameters. Since \(\Psi^{-1}\) is defined using geodesic flow for this existence time, we have shown that \(\Psi^{-1}\) is continuous.

\stepheader{\step: \(\Psi\) is an immersion}
\newcommand{\immersionstepnumber}{\thestepnumber}
Since \(\Psi(-\zeta) = -\Psi(\zeta)\) for all \(\zeta \in \oNM_i\) we may without loss of generality consider only past-directed vectors. Let \(\eta \in \oNM_i\) be past-directed. We will show that the kernel of the tangent map \(\Psi_*(\eta)\) is trivial. We start by finding a hyperplane on which this tangent map is injective. Let \(\Phi_g\) denote the geodesic flow on \(N\M_i\). Let \(T_\partial \colon \oNM_i \to \R\) be the time-to-boundary function from \prettyref{lem:Smooth time to boundary}. Note that the composition of \(\Psi\) with the projection \(\pi_1 \colon (\eta, b) \mapsto \eta\) is the map \(\eta \mapsto \Phi_g(T_\partial(\eta), \eta)\). By \prettyref{lem:Smooth time to boundary} we know that \(T_\partial\) is smooth. It is also a submersion since \({\lambda T_\partial(\lambda X) = T_\partial(X)}\) for all \(\lambda > 0\). This means that its level sets are hypersurfaces in \(\oNM_i\). Let \(\mathfrak L\) be the level set at level \(\tau = T_\partial(\eta)\). This means that for \(\xi \in \mathfrak L\) it holds that \(\pi_1(\Psi(\xi)) = \Phi_g(\tau, \xi)\). Since \(\Phi_g(\tau, \cdot\,)\) is a diffeomorphism for fixed \(\tau\), this means that \(\pi_1 \circ \Psi\) restricted to \(\mathfrak L\) is a diffeomorphism onto its image. Hence the rank of the tangent map \(\Psi_*\) at \(\eta\) must be at least \(\dim(\mathfrak L) = \dim(\oNM_i) - 1\). We will now consider a direction transverse to \(\mathfrak L\).

A particularly nice curve through \(\eta\) transverse to \(\mathfrak L\) is obtained by parallel transport of \(\eta\) along \(\gamma_\eta\). In other words, the curve \(s \mapsto \dot \gamma_\eta(s)\) is transverse to \(\mathfrak L\). Note that \(\pi_1(\Psi(\dot\gamma_\eta(s)))\) is constant and that \({{(\pi_1)}_* \! \circ \Psi_*}\) is injective on \(T_\eta \mathfrak L\), so if the derivative of \(s \mapsto \Psi(\dot\gamma_\eta(s))\) at \(0\) is nonzero, then it is linearly independent of \(\Psi_*(T_\eta\mathfrak L)\), and we will have shown that \(\Psi\) is immersive at \(\eta\). As in the previous case we will compose with a projection. Let \(\pi_2 \colon (\eta, b) \mapsto b\). We will show that the derivative of \(s \mapsto \pi_2(\Psi(\dot\gamma_\eta(s)))\) at \(0\) is nonzero. To do this, we will investigate what happens to the null Weingarten map \(b\) at the boundary when we move the apex of the pre-sky from \(\gamma_\eta(0)\) to \(\gamma_\eta(s)\). The first step consists of using the optical equation for \(b\) to show that \(b\) at the boundary changes if and only if \(b\) changes at any other point along \(\gamma_\eta\) before the apex. The second step is to show that \(b\) changes in a neighborhood of \(\gamma_\eta(0)\) when we move the apex.

\begin{figure}
 \centering{}
 \begin{overpic}[width=0.5\textwidth]{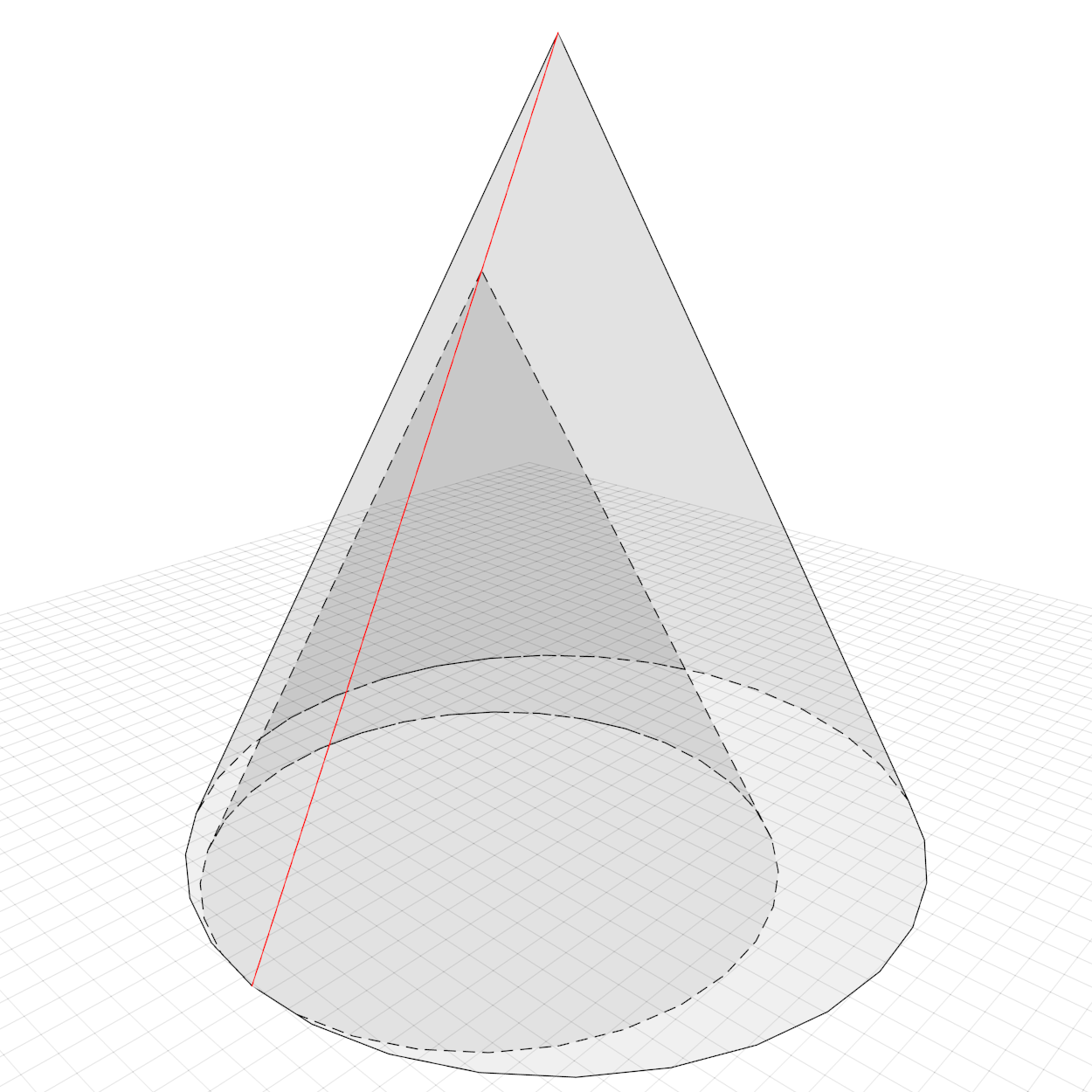}
  \put (52,97) {\(\gamma_\eta(0)\)}
  \put (45,75) {\(\gamma_\eta(s)\)}
  \put (-2,5) {\(\gamma_\eta(T_\partial(\eta))\)}
  \put (95,40) {\(\partial \M_i\)}
 \end{overpic}
 \caption{The past sky shadow of the point \(\gamma_\eta(s)\) differs from the past sky shadow of the point \(\gamma_\eta(s)\) for \(s > 0\). In Step \immersionstepnumber{} of the proof of \prettyref{thm:Past sky shadow data determines conformal class}, the null Weingarten map of (the image under the exponential map of) the lightcone of \(\gamma_\eta(s)\) evaluated at \(\gamma_\eta(t)\) (for \(t > s \geq 0\)) is denoted \(b(t, s)\). What is shown in that step is that the derivative of \(b(T_\partial(\eta), s)\) with respect to \(s\), evaluated at \(s = 0\), is nonzero.}
\label{fig:Nested cones}
\end{figure}

Let \(b(\cdot, s)\) denote the null Weingarten map (along \(\gamma_\eta\)) of the lightcone of the point \(\gamma_\eta(s)\) (or, more precisely, of the image of the lightcone under the exponential map). 
Two such lightcones are shown in \prettyref{fig:Nested cones}.
As before, for each \(s\) the map \(b(\cdot, s)\) satisfies the differential equation
\[b'(t, s) = -{b(t, s)}^2 - R_{\dot\gamma_\eta(t)},\]
where \('\) denotes covariant derivative in direction \(\dot\gamma_\eta(t)\), in other words derivative with respect to \(t\).
We do not have suitable initial data at \(t = s = 0\) for this equation since the trace of \(b\) diverges as \(t \to s\). However, for each \(\epsilon > 0\) and each \(0 \leq s < \epsilon\), we can use \(b(\epsilon, s)\) as initial data, and evolve \(b\) from there: For \(\epsilon\) to be fixed later, let \(\tilde b_s\) be the solution of the initial value problem
\[\begin{cases}
\tilde b_s'(t) = -{\tilde b_s(t)}^2 - R_{\dot\gamma_\eta(t)},\\
\tilde b_s(\epsilon) = b(\epsilon, s).
\end{cases}\]
Let \(\tau = T_\partial(\eta)\). Now \(\pi_2(\Psi(\dot\gamma_\eta(s))) = \tilde b_s(\tau)\). By the theory of ordinary differential equations, the value of \(\tilde b_s(\tau)\) depends diffeomorphically on the initial data \(b(\epsilon, s)\). This means that if \(\frac{d}{ds} b(\epsilon, s) \neq 0\), then \(\frac{d}{ds} \pi_2(\Psi(\dot\gamma_\eta(s))) \neq 0\), which is what is needed to complete the proof of the claim that \(\Psi\) is an immersion. In fact, it is true that \(\frac{d}{ds} b(\epsilon, s) \neq 0\) at \(s = 0\) for all sufficiently small \(\epsilon > 0\). To see this, take the trace of the differential equation for \(b\), letting \(\theta(t, s) = \tr(b(t, s))\):
\[\dot\theta(t, s) = -\frac{{\theta(t, s)}^2}{n-2} - \tr(R_{\dot\gamma_\eta(t)}).\]
Recall that what happens for \(t > s\) as \(t \to s\) for fixed \(s\) is that \(\theta(t, s) \to -\infty\). Since the curvature term \(R\) and all its derivatives are locally bounded, this means that
\[\theta(t, s) = \frac{1}{s - t} + e(s, t)\]
for some error term \(e\) with locally bounded derivatives (when \(s\) and \(t\) are close to \(0\)), and hence
\[\frac{d}{ds} \theta(t, s) = -\frac{1}{{(s - t)}^2} + \frac{d}{ds} e(s, t) \neq 0\]
for all \(t\) and \(s\) sufficiently close to zero. In particular, if we let \(\epsilon > 0\) be sufficiently small and let \(0 \leq s < \epsilon\) then
\[\frac{d}{ds} \theta(\epsilon, s) \neq 0.\]
This shows that \(\frac{d}{ds} \pi_2(\Psi(\dot\gamma_\eta(s))) \neq 0\), completing the proof of the claim.

\stepheader{Conclusion}
Recall that we are considering two manifolds \(\M_1\) and \(\M_2\) with interiors denoted \(M_1\) and \(M_2\) and with the same past sky shadow data, and that we constructed a fiber bundle \(J\) and a set \(\Omega \subseteq J\) in the beginning of the proof. We have shown above that \(\Omega\) is an embedded submanifold of \(J\) and that each \(\oNM_i\) is diffeomorphic to \(\Omega\) by a map \(\Psi_i\). 
(To see the connection with the idea in \prettyref{sec:Reconstruction of smooth structures} we can define \(\Theta_i\) by composing \(\Psi_i^{-1}\) with the projection \(p_i \colon \oNM_i \to M_i\). We will not use the notation \(\Theta_i\), but rather work directly with the composition.)
We would now like to find a map \(\psi\) making the following diagram commute:
\[\begin{tikzpicture}[every node/.style={midway}]
\matrix[column sep={5em,between origins}, row sep={2em}] at (0,0) {
	\node(NM1) {\(\oNM_1\)}; & \node(Omega) {\(\Omega\)}; & \node(NM2) {\(\oNM_2\)}; \\

	\node(M1) {\(M_1\)}; &
	; &
	\node(M2) {\(M_2\)};\\
};

\draw[->]					(NM1) 	-- (Omega)	node[anchor=south]	{\(\Psi_1\)};
\draw[->]					(Omega) -- (NM2)		node[anchor=south]	{\(\Psi_2^{-1}\)};
\draw[->]					(NM1) 	-- (M1)			node[anchor=east] 	{\(p_1\)};
\draw[->]					(NM2)		-- (M2)			node[anchor=east]		{\(p_2\)};
\draw[dashed,->]	(M1)		-- (M2)			node[anchor=south]	{\(\psi\)};
\end{tikzpicture}\]

To show that such a map exists, it is sufficient to show that if \(p_1(\eta) = p_1(\xi)\) then \(p_2(\Psi_2^{-1}(\Psi_1(\eta))) = p_2(\Psi_2^{-1}(\Psi_1(\xi)))\). This, however, is almost tautological: That \(p_1(\eta) = p_1(\xi)\) means that \(\eta\) and \(\xi\) are based at the same point. This means that \(\Psi_1(\eta)\) and \(\Psi_1(\xi)\) come from the same past sky shadow. This in turn means that \(\Psi_2^{-1}(\Psi_1(\eta))\) and \(\Psi_2^{-1}(\Psi_1(\xi))\) are different vectors at the same point, so that \(p_2(\Psi_2^{-1}(\Psi_1(\eta))) = p_2(\Psi_2^{-1}(\Psi_1(\xi)))\).
Since we have assumed that every interior point of \(\M_i\) lies on a past-inextendible lightlike geodesic which intersects the boundary transversely, the projections \(p_i\) are surjective. Hence there is a unique function \(\psi\) making the above diagram commute. Since the \(\oNM_i\) are open submanifolds of \(NM_i\), the projections are surjective submersions. This means that \(\psi\) is smooth. Repeating the argument with the roles of \(M_1\) and \(M_2\) reversed gives a smooth inverse to \(\psi\), so \(\psi\) is a diffeomorphism. Since the diagram commutes and the fibers of \(\oNM_i\) are nonempty open submanifolds of the fibers of \(NM_i\), the diffeomorphism \(\psi\) has to be conformal.
By an application of~\cite[Theorem 5.3]{Chrusciel10} as in Step~\ref{step:Conclude isometry including boundary} in the proof of \prettyref{thm:Fireworks data determines isometry class}, the manifolds \(\M_1\) and \(\M_2\) including boundaries are conformal, completing the proof.
\end{proof}

\section{Broken lightlike scattering rigidity}
\label{sec:Broken lightlike scattering rigidity}
This section concerns our final rigidity theorem, using \enquote{broken lightlike scattering data}. This data is defined similarly to the fireworks data and the broken timelike lens data, except using only lightlike geodesics. The fact that the geodesics are lightlike means that their lengths are always zero, so there is no point in including these lengths in the data.
\begin{defn}
Let \(\M\) be a time-oriented Lorentzian manifold with boundary \(\partial \M\). The \emph{broken lightlike scattering data} of \(\M\) is the set of pairs \((\xi, \eta) \in \TdM \times \TdM\) such that
\begin{itemize}
\item \(\xi\) is future-directed and lightlike,
\item \(\eta\) is future-directed and lightlike,
\item the geodesic starting with \(\xi\) intersects the geodesic ending with \(\eta\).
\end{itemize}
\end{defn}

We will obtain the rigidity theorem in this section as a corollary of \prettyref{thm:Past sky shadow data determines conformal class} by constructing the past sky shadow data from the broken lightlike scattering data. To do this we will also need the \enquote{lightlike scattering data}, which we call \enquote{unbroken lightlike scattering data} for emphasis. This data is somewhat simpler than the broken lightlike scattering data in that it encodes information about lightlike geodesics which are not allowed to have any breakpoints.
\begin{defn}
Let \(\M\) be a time-oriented Lorentzian manifold with boundary \(\partial \M\). The \emph{(unbroken) lightlike scattering data} of \(\M\) is the set of pairs \((\xi, \eta) \in \TdM \times \TdM\) such that
\begin{itemize}
\item \(\xi\) is future-directed and lightlike,
\item \(\eta\) is future-directed and lightlike,
\item there is a geodesic starting with \(\xi\) and ending with \(\eta\).
\end{itemize}
\end{defn}

The manifolds \(\M_{T}\) of \prettyref{ex:Periodic spacetime} show that the broken and unbroken lightlike scattering data is in general not sufficient to determine the conformal class of a manifold. However, if we assume that lightlike geodesics never refocus, in other words that no pair of lightlike geodesics intersect in more than one point, then we can use the broken lightlike scattering data together with the unbroken lightlike scattering data to construct the past sky shadow data. If we, in addition, have the conditions necessary for applying \prettyref{thm:Past sky shadow data determines conformal class}, then we do get the conclusion that the conformal class is uniquely determined by the data. This is what we will do now.
\begin{prop}\label{prop:Broken lightlike scattering data determines past sky shadow data}
Let \((\M_1, g_1)\) and \((\M_2, g_2)\) be strongly causal, compact, time-oriented Lorentzian manifolds of dimension \(n \geq 3\) with boundary. Suppose that they are such that 
\begin{itemize}
\item no two points are connected by two different lightlike geodesics,
\item all lightlike geodesics are achronal.
\end{itemize}
Suppose that \(\M_1\) and \(\M_2\) have isomorphic broken and unbroken lightlike scattering data in the sense that \(\TdM_1 \to \partial \M_1\) and \(\TdM_2 \to \partial \M_2\) are conformal as smooth vector bundles with Lorentzian metrics, by a conformal vector bundle isomorphism which
\begin{itemize}
\item identifies the tangent bundles of the boundaries, that is identifies \(T\partial \M_1\) with \(T\partial \M_2\),
\item identifies the broken lightlike scattering data of \(\M_1\) with the broken lightlike scattering data of \(\M_2\),
\item identifies the unbroken lightlike scattering data of \(\M_1\) with the unbroken lightlike scattering data of \(\M_2\).
\end{itemize}
Then \(\M_1\) and \(\M_2\) have isomorphic past sky shadow data.
\end{prop}
\begin{proof}
Let \(\M\) be one of the manifolds \(\M_1\) and \(\M_2\).
Let \(S\) denote the unbroken lightlike scattering data of \(M\), and let \(B\) denote the broken lightlike scattering data of \(M\). We begin by using \(B\) and \(S\) to determine if the future-directed geodesics starting with some lightlike vectors \(\xi_1\) and \(\xi_2\) intersect. This happens precisely when \((\xi_1, \eta_2) \in B\) and \((\xi_2, \eta_1) \in B\) where \(\eta_1\) and \(\eta_2\) are such that \((\xi_i, \eta_i) \in S\). Let us denote this relation by \(I\), and use the notation \(\xi_1 I \xi_2\).
We will show that \(\Sigma\) is a past sky shadow if and only if it is a maximal set with the property that \(\xi I \xi'\) for all \(\xi, \xi' \in \Sigma\).

Suppose first that \(\Sigma\) is the past sky shadow of some point \(q \in M\). Then the elements of \(\Sigma\) are precisely those vectors which give rise to geodesics through \(q\), so \(\xi I \xi'\) for all \(\xi, \xi' \in \Sigma\). There is no proper superset of \(\Sigma\) with this property, since an additional geodesic would have to intersect the others in points different from \(q\), which would contradict the achronality of at least one of the geodesics.

Conversely, suppose that \(\Sigma\) is of the form described above. Since \(\Sigma\) is maximal, we know that it contains at least two elements \(\xi_1\) and \(\xi_2\). The geodesics \(\gamma_i\) starting with \(\xi_i\) can intersect in at most one point by hypothesis. They must intersect in some point, since \(\xi_1 I \xi_2\). Let the point of intersection be \(q\). If \(\xi_3 I \xi_1\) and \(\xi_3 I \xi_2\) then the geodesic starting with \(\xi_3\) must pass through \(q\) since the three geodesics are achronal. Since \(\Sigma\) is maximal, it must be all of the past sky shadow of \(q\).
\end{proof}

\begin{thm}\label{thm:Broken lightlike scattering rigidity}
Let \((\M_1, g_1)\) and \((\M_2, g_2)\) be strongly causal, compact, time-oriented Lorentzian manifolds of dimension \(n \geq 3\) with boundary. Suppose that they are such that 
\begin{itemize}
\item no two points are connected by two different lightlike geodesics,
\item all lightlike geodesics are achronal,
\item no lightlike geodesic contains a pair of conjugate points, and
\item every point is the future endpoint of a past-inextendible timelike geodesic which intersects the boundary transversely.
\end{itemize}
Suppose that \(\M_1\) and \(\M_2\) have isomorphic broken and unbroken lightlike scattering data in the sense that \(\TdM_1 \to \partial \M_1\) and \(\TdM_2 \to \partial \M_2\) are conformal as smooth vector bundles with Lorentzian metrics, by a conformal vector bundle isomorphism which
\begin{itemize}
\item identifies the tangent bundles of the boundaries, that is identifies \(T\partial \M_1\) with \(T\partial \M_2\),
\item identifies the broken lightlike scattering data of \(\M_1\) with the broken lightlike scattering data of \(\M_2\),
\item identifies the unbroken lightlike scattering data of \(\M_1\) with the unbroken lightlike scattering data of \(\M_2\).
\end{itemize}
Then \(\M_1\) and \(\M_2\) are conformal.\\
In other words, the conformal class of any strongly causal, compact, time-oriented Lorentzian manifold of dimension \(n \geq 3\) with boundary with the above conditions on lightlike geodesics is uniquely determined by the broken lightlike scattering data together with the unbroken lightlike scattering data.
\end{thm}
\begin{proof}
Combine \prettyref{prop:Broken lightlike scattering data determines past sky shadow data} and \prettyref{thm:Past sky shadow data determines conformal class}.
\end{proof}

\appendix
\section{Some definitions}\label{app:Assorted definitions}
The following definitions can be found in~\cite[p. 53-54]{BeemEhrlich},~\cite[p. 59]{BeemEhrlich} and~\cite[Definition 11 in Chapter 14]{ONeill}, respectively. They are standard apart from that we have defined causal convexity and strong causality for manifolds with possibly nonempty boundary.
\begin{defn}
A \emph{convex neighborhood} in a Lorentzian manifold is a neighborhood such that any two points can be joined by a unique geodesic in the neighborhood. 
\end{defn}
\begin{defn}
A \emph{causally convex neighborhood} in a Lorentzian manifold with boundary is a neighborhood such that no causal curve intersects it twice. 
\end{defn}
\begin{defn}
A Lorentzian manifold with boundary is \emph{strongly causal} if for each point \(p\) and for each neighborhood \(U\) of \(p\), there is a causally convex neighborhood \(p \in V \subseteq U\).
Strong causality implies causality (that there are no closed causal curves) and is implied by global hyperbolicity. 
\end{defn}

\section{Null geometry}
\label{app:Null geometry}
For completeness, we summarize the properties of smooth null hypersurfaces which are used in the paper. As references for null geometry, we suggest~\cite{Galloway04} and~\cite{Kupeli87}.

\subsubsection*{Null hypersurfaces and null Weingarten maps}
A \emph{null hypersurface} is a hypersurface whose normal vector field is lightlike. The constructions in this appendix are local, so they work even if the hypersurfaces are only immersed instead of embedded. We will sometimes use the word \enquote{immersed} for emphasis. 
Let \(\mathcal H\) be a null hypersurface. There is no canonical choice of normal vector field on a null hypersurface, so fix one particular choice \(K\). The \emph{null Weingarten map} of \(\mathcal H\) with respect to the choice of normal vector field \(K\) is the map \(b \colon \qextraspace{T\mathcal H}{K} \to \qextraspace{T\mathcal H}{K}\) defined by letting \(b(X) = [\nabla_{\hat X} K]\) where \(\hat X\) is any representative of \(X \in \qextraspace{T\mathcal H}{K}\). The brackets denote taking equivalence classes. A computation shows that \(b(X)\) does not depend on the representative \(\hat X\) chosen. Perhaps more surprisingly, \(b(X)\) evaluated at a point \(p\) does not depend on the vector field \(K\) outside of \(p\). If \(K_p = \eta\), we use \(b_\eta\) to denote the null Weingarten map with respect to \(K\) evaluated at \(p\). It can easily be seen that \(b_{\lambda \eta} = \lambda b_\eta\) for \(\lambda \neq 0\). In other words, rescaling the normal vector field by \(\lambda \neq 0\) rescales \(b\) by \(\lambda\).

The null Weingarten map does not depend on the behavior of the null hypersurface in the null direction, in the sense that it can be computed from a spacelike slice of the hypersurface together with a lightlike normal vector field along the slice.

When \(\eta\) is a lightlike vector we use \(\R\eta\) and \(\eta^\perp\) to denote the spaces of vectors which are tangent to \(\eta\) and orthogonal to \(\eta\), respectively. Since \(\eta\) is lightlike we have \(\R\eta \subset \eta^\perp\), and we denote the vector space quotient of these spaces by \(\qextraspace{\eta^\perp}{\eta}\).

Similarly, when \(\gamma\) is an immersed curve we use the notation \(T\gamma\) for the tangent bundle of \(\gamma\) (viewed as a vector bundle along \(\gamma\)), and \(T^\perp\gamma\) for the normal bundle of \(\gamma\). If \(\gamma\) is a lightlike geodesic contained in a null hypersurface \(\mathcal H\), then \(T\mathcal H = T^\perp\gamma\) so the null Weingarten map of \(\mathcal H\) is an endomorphism on \(\qextraspace{T^\perp\gamma}{T\gamma}\).

\subsubsection*{The optical equation}
Let \(\mathcal H\) be an immersed null hypersurface and let \(\gamma\) be a lightlike geodesic taking its values in \(\mathcal H\). Then the null Weingarten map \(b_{\dot\gamma(t)}\) along \(\gamma\) satisfies the Riccati equation
\[b_{\dot\gamma(t)}' = -b_{\dot\gamma(t)}^2 - R_{\dot\gamma(t)}\]
where \(R_Y\) is the \((1,1)\)-tensor on \(\qextraspace{T^\perp\gamma}{T\gamma}\) constructed from the Riemann curvature tensor \(\Riem\) by
\[R([X]) = [\Riem(X, Y, Y)].\]
The derivative \(b_{\dot\gamma(t)}'\) is the Levi-Civita covariant derivative of \(b_{\dot\gamma(t)}\) in direction \(\dot\gamma(t)\), which is well-defined even after taking equivalence classes.
Following~\cite{CDGH01}, we call the equation the \emph{optical equation}.

\subsection*{Acknowledgements}
I would like to thank Mattias Dahl and Matti Lassas for many helpful discussions and advice concerning this project.
In addition, I would like to thank Institut Henri Poincar\'e for providing support and a stimulating working environment during the Mathematical General Relativity Trimester 2015, where this work was finished. 

\bibliographystyle{abbrv}
\bibliography{References}

\end{document}